\documentclass[11pt,leqno]{amsart}
\usepackage{anysize}
\marginsize{3.5cm}{3.5cm}{3cm}{3cm}
\usepackage [latin1]{inputenc}
\usepackage[]{hyperref}
\hypersetup{
    colorlinks=true,       
    linkcolor=red,          
    citecolor=blue,        
    filecolor=magenta,      
    urlcolor=cyan           
}

\usepackage{amsmath}
\usepackage{amsfonts,amssymb}
\usepackage{enumerate}
\usepackage{mathrsfs}
\usepackage{amsthm}

\theoremstyle{plain}
\newtheorem{theo}{Theorem}[section]
\newtheorem*{theo*}{Theorem}
\newtheorem{prop}[theo]{Proposition}
\newtheorem{lemm}[theo]{Lemma}
\newtheorem{coro}[theo]{Corollary}

\newtheorem{defi}[theo]{Definition}
\theoremstyle{definition}
\newtheorem{rema}[theo]{Remark}

\DeclareMathOperator{\cnx}{div}
\DeclareMathOperator{\cn}{div}
\DeclareMathOperator{\RE}{Re}

\DeclareMathOperator{\Op}{Op}

\DeclareMathOperator{\diff}{d}
\DeclareSymbolFont{pletters}{OT1}{cmr}{m}{sl}
\DeclareMathSymbol{s}{\mathalpha}{pletters}{`s}

\def\ah{\arrowvert_{y=h}}

\def\ba{\begin{align}}
\def\bad{\begin{aligned}}
\def\be{\begin{equation}}
\def\ea{\end{align}}
\def\ead{\end{aligned}}
\def\ee{\end{equation}}
\def\e{\eqref}

\def\dsigma{\diff \! \sigma}
\def\deta{\diff \! \eta}
\def\dt{\diff \! t}
\def\dx{\diff \! x}
\def\dxi{\diff \! \xi}

\def\dy{\diff \! y}

\def\fract{\frac{\diff}{\dt}}
\def\fractt{\frac{\diff^2}{\dt^2}}

\def\cnxy{\cn_{x,y}}

\def\defn{\mathrel{:=}}

\def\Deltayx{\Delta_{x,y}}

\def\eps{\varepsilon}

\def\la{\left\vert}
\def\lA{\left\Vert}
\def\le{\leq}
\def\les{\lesssim}
\def\leo{}

\def\mez{\frac{1}{2}}
\def\partialx{\nabla}

\def\ra{\right\vert}
\def\rA{\right\Vert}

\def\tdm{\frac{3}{2}}

\def\xN{\mathbf{N}}
\def\xR{\mathbf{R}}

\def\xT{\mathbf{T}}

\def\cG{\mathcal{G}}

\numberwithin{equation}{section}

\begin{document}

\pagestyle{plain}

\date{}
\title{Lyapunov functions, Identities and the Cauchy problem for the Hele-Shaw equation} 
\author{Thomas Alazard, Nicolas Meunier, Didier Smets}

\maketitle

\begin{center}{\large \textbf{Abstract}}
\end{center}
\vspace{5mm}
This article is devoted to the study of the Hele-Shaw equation. 
We introduce an approach inspired by the water-wave theory. 
Starting from a reduction to the boundary, introducing the 
Dirichlet to Neumann operator and exploiting various cancellations, we exhibit 
parabolic evolution equations for the horizontal and vertical traces of the velocity on the free surface. 
This allows to quasi-linearize the equations in a natural way. 
By combining these exact identities with convexity inequalities, 
we prove the existence of hidden Lyapunov functions of different natures. 
We also deduce from these identities and previous works on the water wave problem 
a simple proof of the well-posedness of the Cauchy problem. 
The analysis contains two side results of independent interest. Firstly, we give a principle 
to derive estimates for the modulus of continuity of a PDE under general assumptions on the flow. 
Secondly we prove and give applications of a convexity inequality for the Dirichlet to Neumann operator.

\section{Introduction}

Consider a $d$-dimensional 
fluid domain $\Omega$,   
located underneath a free surface $\Sigma$ 
given as a graph, so that at time $t\ge 0$
\begin{align*}
\Omega(t)&=\{ (x,y) \in \xT^{n}\times \xR\,;\, y < h(t,x)\}, \quad n=d-1,\\
\Sigma(t)&= \{ (x,y) \in \xT^{n}\times \xR\,;\, y = h(t,x)\},
\end{align*}
where $\xT^n$ denotes a $n$-dimensional torus (our analysis applies also when $\xT^n$ is replaced by $\xR^n$). 
In the Eulerian coordinate system, the unknowns are the velocity field 
$v$, the scalar pressure $P$ and the free surface elevation~$h$. 
The Hele-Shaw equation described the dynamics of an incompressible 
liquid whose velocity obeys Darcy's law, so that
\be\label{Darcy}
\cnxy v=0 \quad\text{ and }\quad v=-\nabla_{x,y} (P+gy) \quad \text{in }\Omega,
\ee
where $g$ is the acceleration of gravity. These equations are supplemented by the boundary conditions:
\be
\left\{
\begin{aligned}
&P=0 && \text{on }y=h(t,x),\\
&\partial_t h=\sqrt{1+|\partialx h|^2} \, v\cdot n && \text{on }y=h(t,x),\\
\end{aligned}
\right.
\ee
where $\nabla=\nabla_x$ and $n$ is the outward unit normal to $\Sigma$, given by
$$
n=\frac{1}{\sqrt{1+|\nabla h|^2}} \begin{pmatrix} -\nabla h \\ 1 \end{pmatrix}.
$$

There are many possible ways to study the Hele-Shaw equation: to mention a few approaches 
we quote various PDE methods based on $L^2$-energy estimates (see 
the works of Chen~\cite{Chen-ARMA-1993}, C{\'o}rdoba, C{\'o}rdoba and Gancedo \cite{CCG-Annals}, Kn{\"u}pfer and Masmoudi~\cite{Knupfer-Masmoudi-ARMA-2015}, 
G{\"u}nther and Prokert~\cite{Gunther-Prokert-SIAM-2006}, Cheng, Granero-Belinch{\'o}n and Shkoller~\cite{Cheng-Belinchon-Shkoller-AdvMath}), there are also 
methods based on functional analysis tools and maximal estimates 
(see Escher and Simonett~\cite{Escher-Simonett-ADE-1997}, the results reviewed in the book 
by Pr\"uss and Simonett~\cite{Pruss-Simonett-book} and Matioc~\cite{Matioc1,Matioc2}) or methods using 
harmonic analysis tools and contour integrals (see the numerous 
results reviewed in the survey papers by Gancedo~\cite{GancedoSEMA} or Granero-Belinch{\'o}n and Lazar~\cite{GraneroLazar}). 
For the related Muskat equation (a two-phase Hele-Shaw problem), maximum principles have played a key role to study 
the Cauchy problem, see \cite{CCFG-ARMA-2016,CGSV-AIHP2017,Cameron-APDE,Cordoba-Lazar-H3/2} following 
the pioneering work of Constantin, C\'{o}rdoba, Gancedo, 
Rodr\'{i}guez-Piazza and Strain~\cite{CCGRPS-AJM2016}. 
Such maximum principles have been obtained for general viscosity solutions of the Hele-Shaw equation 
by Kim~\cite{Kim}, see also the recent work of Chang-Lara, Guillen and Schwab~\cite{ChangLaraGuillenSchwab}.

In this paper, we introduce 
another approach inspired by the analysis of the water-wave equations: 
we use the Dirichlet to Neumann operator to reduce the Hele-Shaw equation to an equation 
on the free surface and then quasi-linearize the equation thus obtained. To do so, we begin by introducing the potential
$$
\phi=P+gy.
$$
Since the velocity $v$ is divergence free, $\phi$ is harmonic, that is $\Delta_{x,y} \phi=0$. 
Consequently, $\phi$ 
is fully determined by the knowledge of its trace at the free surface, which is $\phi\arrowvert_{y=h}=g h$ since 
$P\arrowvert_{y=h}=0$. 
This explains that the problem can be written as an evolution equation 
involving only the unknown $h$. To write this equation, 
we need the Dirichlet to Neumann operator. This operator maps a function defined on the free surface 
to the normal derivative of its harmonic extension. Namely, for any function 
$\psi=\psi(t,x)$, consider its harmonic extension $\varphi=\varphi(t,x,y)$ solution to 
\be\label{defi:varphipsi}
\Delta_{x,y}\varphi=0 \quad \text{in }\Omega(t),\quad \varphi(t,x,h(t,x))=\psi(t,x).
\ee
Then the Dirichlet to Neumann operator $G(h)$ is defined by 
\begin{align*}
(G(h) \psi)  (t,x)&=
\sqrt{1+|\partialx h|^2}\,
\partial _n \varphi\arrowvert_{y=h(t,x)}\\
&=(\partial_y \varphi)(t,x,h(t,x))-\partialx h (t,x)\cdot (\partialx \varphi)(t,x,h(t,x)).
\end{align*}
Now observe that Darcy's law implies that
$v\cdot n=-\partial_n \phi$. Since $\phi=P+gy$ is the harmonic extension of $g h$, we conclude that 
$$
\sqrt{1+|\partialx h|^2}\, v\cdot n=-G(h)h.
$$
Consequently, $h$ solves the evolution equation
\begin{equation*}
\partial_{t}h+g G(h)h=0.
\end{equation*}
Here the modulus of the constant $g$ is irrelevant since one can always modify it by rescaling the equation in time. 
Assuming that $g>0$, we obtain the following evolution equation for $h$,
\begin{equation}\label{n10}
\partial_{t}h+G(h)h=0.
\end{equation}
This equation is analogous to the Craig-Sulem-Zakharov 
equation in water-wave theory (see~\cite{Zakharov1968,CrSu,Bertinoro}). 
Our first goal is to show how various results developed in the study of the water wave problem 
could be used to study the Cauchy problem for the Hele-Shaw equation. 
Our second and main goal is to find various identities and Lyapunov functionals for the Hele-Shaw equation. 

This paper contains various complementary results whose statements are gathered in the next section to highlight the links between them. They are of different kinds:
\begin{enumerate}
\item \textbf{Identities and the Cauchy problem:} we derive several new exact equations for the Hele-Shaw equation (see Proposition~\ref{Prop:eqBV}). 
Moreover, we deduce a simple proof of the well-posedness of the Cauchy problem in $H^s$ with $s>1+n/2$ in any dimension $n$, by combining the above mentioned identities with 
the paradifferential analysis of the Dirichlet-to-Neumann operator introduced in \cite{AM,ABZ1,ABZ3} (see Theorem~\ref{T:Cauchy}).
\item \textbf{Lyapunov functionals:} this is the most original part of this work. 
We derive several hidden decaying functionals which are of different natures. 
Firstly  we derive by an abstract general principle of independent interest 
a maximum principle for the slope. We also prove the same result by an $L^2$-type energy estimate which allow 
us to prove: $i)$ a new maximum principle for the time derivative, $ii)$ $L^p$-decay estimates for some special 
derivatives. As an application, we deduce a third maximum principle which gives a maximum principle 
for the inverse of the Rayleigh coefficient. 
Eventually, we obtain new Lyapunov functionals which give control of a higher order energy. 
\end{enumerate}

\section{Main results}

\subsection{Cauchy problem}

The main goal of this paper is to find exact identities and Lyapunov functionals for the Hele-Shaw equation. 
As a by-product of this analysis, we shall obtain a simple proof of the well-posedness of the Cauchy problem. 
We begin by the latter result, since it justifies the existence of the regular solutions we will consider. 

As recalled in the introduction, the Cauchy problem for the Hele-Shaw equation has been studied in three different cases: for 
weak solutions, for 
viscosity solutions and also for classical solutions. Here we are interested in classical solutions with initial data 
in Sobolev spaces. Let us recall that 
$H^s(\xT^n)$ is the Sobolev space of periodic functions $h$ such that $(I-\Delta)^{s/2}h$ belongs to 
$L^2(\xT^n)$, where $(I-\Delta)^{s/2}$ is the Fourier multiplier with symbol $(1+\la\xi\ra^2)^{s/2}$. 
Cheng, Granero-Belinch{\'o}n and Shkoller~\cite{Cheng-Belinchon-Shkoller-AdvMath} studied the Cauchy problem in a very general setting. In particular, their results show that the Cauchy problem for the 
Hele-Shaw equation is well-posed 
for initial data in $H^s(\xT)$ with $s\ge 2$. We will prove 
that the same result holds for any $n\ge 1$ and any $s>n/2+1$. A key remark here is that the proof will be in fact 
a straightforward consequence of identities obtained later in this paper and the easy part of the paradifferential analysis in \cite{AM,ABZ1,ABZ3}. 
(We refer the reader to \cite{Matioc1,Matioc2,Cordoba-Lazar-H3/2,AL} for related results for the Muskat equation, as well as the references therein.)

\begin{theo}\label{T:Cauchy}
Let $n\ge 1$ and consider a real number $s>n/2+1$. For any initial data $h_0$ in $H^s(\xT^n)$, there exists 
a time $T>0$ such that the Cauchy problem
\begin{equation}\label{Hele-Shaw100}
\partial_{t}h+G(h)h=0,\quad h\arrowvert_{t=0}=h_0,
\end{equation}
has a unique solution satisfying
$$
h \in C^0([0,T];H^s(\xT^n))\cap C^1([0,T];H^{s-1}(\xT^n))\cap L^2([0,T];H^{s+\mez}(\xT^n)).
$$
Morevoer, $h$ belongs to $C^\infty((0,T]\times \xT^n)$. 
\end{theo}
\begin{defi}
We say that $h$ is a regular solution to \e{Hele-Shaw100} defined on $[0,T]$ 
if $h$ satisfies the conclusions of the above result.
\end{defi}

\subsection{Maximum principles for the graph elevation}\label{S:maxintro}

The Hele-Shaw equation is a nonlinear parabolic equation, 
so a natural question is to find maximum principles. 
We begin by the simplest question which is to study maximum principle for $h$ itself. 
It is known that 
\be\label{n211}
\sup_{x\in \xT^n}\la h(t,x)\ra\le \sup_{x\in \xT^n}\la h(0,x)\ra.
\ee
On the other hand, by performing an elementary $L^2$-energy estimate, one gets
$$
\int_{\xT^n} h(t,x)^{2}\dx 
\le \int_{\xT^n} h(0,x)^{2}\dx.
$$
We will complement these two results in three directions. Firstly, 
by proving $L^p$ estimates which include the above energy estimate and allow to obtain 
the maximum principle when $p$ goes to $+\infty$.

\begin{prop}
Let $n\ge 1$ and consider an integer $p$ in $\{1\}\cup 2\xN$. 
Assume that $h$ is a regular solution to $\partial_th+G(h)h=0$ defined on $[0,T]$. 
Then, for all time $t$ in $[0,T]$, there holds
\be\label{n210}
\int_{\xT^n} h(t,x)^{2p}\dx
+2\int_0^t\int_{\xT^n}h^p G(h)(h^p)\dx\dt 
\le \int_{\xT^n} h(0,x)^{2p}\dx.
\ee
\end{prop}

Observe that $\int_{\xT^n} \psi G(h)\psi\dx\ge 0$ for any function $\psi$ (see~\e{n217}), so the previous result implies that the 
$L^{2p}$-norm decays. Then one may 
deduce \e{n211} from~\e{n210} 
by arguing that the $L^\infty$-norm of $h$ is the limit of its $L^{2p}$-norms when $p$ goes to $+\infty$. 

We shall improve the maximum principle~\e{n210} to a 
comparison principle.
\begin{prop}\label{prop:orderingintro}
Let $h_1,h_2$ be two regular solutions of the Hele-Shaw 
equation $\partial_t h + G(h)h=0$ defined on the same time interval 
$[0,T]$, such that, initially,
$$
h_1(0,\cdot) \leq h_2(0,\cdot).
$$ 
Then
$$
h_1(t,\cdot) \leq h_2(t,\cdot)
$$
for all $t \in [0,T]$.    
\end{prop}

Eventually, we will prove that the square of the $L^2$-norm decays in a convex manner. 
To do so, we prove the somewhat surprising result that $\int_{\xT^n} hG(h)h \dx$ is 
a Lyapunov function.

\begin{prop}\label{P:energyconvex}
Let $n\ge 1$. 
For any regular solution $h$ of the Hele-Shaw equation, there holds
\begin{align*}
&\mez\fract\int_{\xT^n} h^2\dx +\int_{\xT^n} h G(h)h \dx =0,\\[1ex]
&\fract\int_{\xT^n} h G(h)h \dx+\int_{\xT^n} a(h_t^2+|\nabla h|^2)\dx=0.
\end{align*}
where $a=a(t,x)$ is a positive function (the Rayleigh-Taylor coefficient defined in~\e{defi:RT}). 
Consequently, 
$$
\fract\int_{\xT^n} h(t,x)^2\dx\le 0\quad\text{and}\quad\fractt\int_{\xT^n} h(t,x)^2\dx\ge 0.
$$
\end{prop}

\subsection{Maximum principle for modulus of continuity}

We are interested in giving maximum principles 
for the derivatives of $h$. These bounds are interesting 
since they involve quantities which are scaling invariant. 
In this direction, 
we begin by recalling the following result.

\begin{prop}[from~\cite{Kim,ChangLaraGuillenSchwab}]\label{prop:bornederiv}
Let $n\ge 1$ and assume that $h$ is a regular 
solution to $\partial_th+G(h)h=0$ defined on $[0,T]$. 
Then, for all time $t$ in $[0,T]$,
\be\label{n501}
\sup_{x\in \xT^n}\la \nabla h(t,x)\ra\le \sup_{x\in \xT^n}\la \nabla h(0,x)\ra.
\ee
\end{prop}

We shall provide later a generalization of this result (see Theorem~\ref{NeatMax}). 
In this paragraph we give in details an alternative proof and also a slight generalization which we believe is of independent interest, since it relies on a general principle which could be used in a broader context. Indeed, this proof 
relies only on a comparison principle at the level of functions, as given by Proposition \ref{prop:orderingintro}, with an abstract result pertaining to classes of monotone mappings which are equivariant under suitable group actions. We first explain the latter in its broader framework in order to better highlight the properties at play.   

Let $(X,d)$ be a metric space.  

\begin{defi} A non decreasing function $\omega:\ (0,+\infty) \to (0,+\infty)$ is a modulus of continuity for a function $f: \ X \to \xR$ if and only if 
\[
  |f(x_1)-f(x_2)| \leq \omega(d(x_1,x_2)),\qquad\forall\ x_1,x_2 \in X.
\]
\end{defi}

\noindent
In the sequel we assume that $\cG$ is a group acting on $X$ and which satisfies the following property 
\[
  \forall x_1,x_2 \in X,\ \exists G\in \cG \text{ such that } \left\{ 
  \begin{array}{l}
    G(x_1) = x_2,\\[2pt]
    d(x,G(x)) \leq d(x_1,x_2), \ \forall x \in X.
  \end{array}
  \right.\leqno{(H_0)}
\]
The action of $\cG$ on $X$ induces an action of $\cG$ on $\xR^X$ classically defined by
\[
  G(f)(x) := f(G^{-1}(x)),\qquad\forall\: f\in \xR^X,\: \forall x \in X,
\]
where $G^{-1}$ denotes the inverse of $G$ in $\cG.$ 

\begin{lemm}\label{lem:modulus} Let $F \subseteq \xR^X$ be a $\cG$-invariant vector space which contains the constants,
  and suppose that $\Phi:\: F \to F$ is a mapping which satisfies:
  \begin{align*}
    1)\ & \Phi(f_1) \leq \Phi(f_2), &&  \forall f_1 \leq f_2 \in F         && \text{(monotonicity)},\\
    2)\ & \Phi(G(f)) = G(\Phi(f)),  &&  \forall f \in F,\: \forall G \in \cG&& \text{($\cG$-equivariance)},\\
    3)\ & \Phi(f+c) = \Phi(f)+c,    && \forall f \in F,\: \forall c \in \xR && \text{(equivariance through constants).}
  \end{align*}
  Then, whenever $f\in F$ and $\omega$ is a modulus of continuity for $f$, $\omega$ is also a modulus of
  continuity for $\Phi(f).$ 
\end{lemm}
\begin{proof}
  Let $x_1, x_2$ be arbitrary points in $X$, and let $G$ be given by assumption $(H_0)$ for that
  specific choice of $x_1,x_2.$ The function $\bar{f} := G(f) + \omega(d(x_1,x_2))$ belongs to $F$
  (by assumption on the latter) 
  and satisfies $f\leq \bar{f}.$ Indeed, since $\omega$ is a modulus of continuity for $f$, for
  an arbitrary $x \in X$  we have
  \[
    f(G(x)) \geq f(x) - \omega(d(x,G(x)) \geq f(G(x)) - \omega(d(x_1,x_2)), 
  \]
  where for the last inequality we have used the monotonicity of $\omega$ combined with assumption $(H_0).$ 
  From the monotonicity of $\Phi,$ it follows that $\Phi(f)\leq \Phi(\bar{f}).$ On the other hand,
  from both equivariances of $\Phi$ we obtain
  \[
    \Phi(\bar{f})(x) = \Phi(f)(G^{-1}(x)) + \omega(d(x_1,x_2)),\qquad \forall x \in X.  
  \]
  Specified at the point $x=x_2$, the previous identity together with the inequality $\Phi(f)\leq \Phi(\bar{f})$ yield 
  \[
    \Phi(f)(x_2) \leq \Phi(f) (x_1) + \omega(d(x_1,x_2)), 
  \]
  from which the conclusion follows by arbitrariness of $x_1$ and $x_2.$
\end{proof}

Proposition \ref{prop:bornederiv} is an immediate consequence of the following.

\begin{prop}\label{prop:modcont}
Let $n\ge 1$ and consider a regular 
solution $h$ to $\partial_th+G(h)h=0$ defined on $[0,T]$. 
Then, whenever $\omega$ is a modulus of continuity for $h(0,\cdot)$, $\omega$ is also a modulus of
continuity for $h(t,\cdot)$, for any $t \in [0,T]$.
\end{prop}
\begin{proof}
We apply Lemma \ref{lem:modulus} with $X = \xT^n$, $F = C^\infty(\xT^n)$ and $\Phi$ being the solution map for the Hele-Shaw equation from time $0$ to some fixed arbitrary time $t\geq 0.$ The group $\cG$ acting on $X$ is simply $\xR^n$ and the action is by translation. The fact that Assumption 1) in Lemma \ref{lem:modulus} is satisfied is precisely the statement of Proposition \ref{prop:orderingintro}. Assumption 2) follows from the invariance of the Hele-Shaw equation under translation in the space variables, and assumption 3) is an easy consequence of our setting with an infinite depth.
\end{proof}

\subsection{Identities and Lyapunov functionals}

Proposition~\ref{prop:bornederiv} gives a maximum principle for the 
$L^\infty$-norm of the spatial derivatives. Such results are quite classical for parabolic equations. 
We shall see in this section that there are other hidden Lyapunov functions which, to the authors knowledge, cannot 
be derived from general principles for parabolic equations. These Lyapunov functions will allow us to control other derivatives.

The main difficulty is to find good derivatives, for which one can form simple evolution equations. Guided by the analysis in Alazard-Burq-Zuily~\cite{ABZ1,ABZ3}, we work with the horizontal and vertical traces 
of the velocity at the free surface:
\be\label{defi:BVphii}
B= (\partial_y \phi)\arrowvert_{y=h},\quad 
V = (\nabla_x \phi)\arrowvert_{y=h}.
\ee
They are given in terms of $h$ by the following formulas (see Proposition~\ref{cancellation}),
\begin{equation}\label{n1201}
B= \frac{G(h)h+\la \partialx h\ra^2}{1+|\partialx  h|^2},
\qquad
V=(1-B)\partialx h.
\end{equation}

\begin{prop}\label{Prop:eqBV}
For regular solutions, the derivatives~$B$ and~$V$ satisfy
\be\label{eqBV}
\left\{
\begin{aligned}
&\partial_t B-V\cdot \nabla B +(1-B)G(h)B=\gamma,\\
&\partial_t V-V\cdot \nabla V +(1-B)G(h)V+\frac{\gamma}{1-B} V=0,
\end{aligned}
\right.
\ee
where $\gamma$ is given explicitly by
\begin{equation}\label{eq:defgamma}
\gamma=\frac{1}{1+|\nabla h|^2}\Big(G(h)\big(B^2+|V|^2\big)-2BG(h)B-2V\cdot G(h)V\Big).
\end{equation}
\end{prop}
The above proposition lies at the heart of our analysis. Indeed, we shall use it to study the Cauchy problem for the Hele-Shaw equation. 
To explain this, we need to introduce another important physical quantity: the Rayleigh--Taylor coefficient
\be\label{defi:RT}
a=-(\partial_y P)\ah=1-B.
\ee
The sign of $a$ dictates the stability of the Cauchy problem. 
In our setting, the well-posedness of the Cauchy problem follows from the fact that 
$a=1-B$ is always positive, so that $aG(h)$ is a positive elliptic operator of order one. 
The latter claim will be made precise in Section~\ref{S:paraDN}. This implies that 
the equations for $B,V$ are parabolic and 
the well-posedness follows. 
Recall that the positivity of $a$ is a well-known property which can be deduced from 
Zaremba's principle (see~\S\ref{S:max}).

We shall also use the equations for $B$ and $V$ to obtain a sharp maximum principle, including the time derivative.

\begin{theo}\label{NeatMax}
Let $n\ge 1$. Consider a positive number $M>0$ and 
a regular solution $h$ of $\partial_th+G(h)h=0$ defined on $[0,T]$. 
Then for any derivative
$$
D\in \{\partial_t,\partial_{x_1},\ldots,\partial_{x_n}\},
$$
if, initially, $\sup_{x\in \xT^n} D h(0,x)\le M$, then, for all time $t$ in $[0,T]$,
$$
\sup_{x\in \xT^n} Dh(t,x)\le M.
$$
\end{theo}

We shall work out two other applications of the equations for $B$ and $V$. 
Namely, we shall prove decay estimates for the $L^p$-norms of the inverse of the 
Rayleigh--Taylor coefficient $a=1-B$ and for the horizontal velocity $V$ when $n=1$.

\begin{theo}\label{T:Bp}
Let $n\ge 1$ and consider a real number $p$ in $[1,+\infty)$. 
Assume that $h$ is a regular solution to $\partial_th+G(h)h=0$ defined on $[0,T]$. 
Then, for all time $t$ in $[0,T]$,
$$
\int_{\xT^n} \frac{\dx}{a(t,x)^p}\le 
\int_{\xT^n} \frac{\dx}{a(0,x)^p}.
$$
Consequently, for all time $t$ in $[0,T]$,
$$
\inf_{x\in\xT^n}a(t,x)\ge \inf_{x\in \xT^n}a(0,x).
$$
\end{theo}

We give two surprising applications of the previous inequality.

\begin{prop}
Let $n\ge 1$ and consider a regular solution $h$ to $\partial_th+G(h)h=0$ defined on $[0,T]$. 
Set
$$
a_0=\inf_{x\in\xT^n} a(0,x)>0.
$$
$i)$ Then, for any time $t\in [0,T]$,
\be\label{n2002}
\int_{\xT^n} \la  \nabla h(t,x)\ra^2\dx \le \frac{\la \xT^n\ra}{a_0}.
\ee
$ii)$ If in addition $G(h_0)h_0\ge -1$, then
\be\label{n2003}
\sup_{x\in \xT^n}\la \nabla h(t,x)\ra\le \min\Big\{\frac{1}{a_0},\sqrt{\frac{2}{a_0}}\Big\}.
\ee
\end{prop}
\begin{proof}
We give the proof here since it is elementary and allows to illustrate several results.

$i)$ The proof relies on the following trick: since
$$
\int_{\xT^n}G(h)h\dx=0, \quad G(h)h=B-V\cdot\nabla h,\quad V=(1-B)\nabla h=a\nabla h,
$$
we have
$$
\int_{\xT^n}B\dx =\int_{\xT^n}V\cdot \nabla h\dx=\int_{\xT^n}a\la \nabla h\ra^2\dx.
$$
Then, we use two ingredients. Firstly, 
the positivity of the Rayleigh-Taylor coefficient $a=1-B$ (see Proposition~\ref{Coro:Zaremba-Taylor}) 
to infer that $B\le 1$; and secondly we use the above theorem: 
$\inf_x a(t,x)\ge \inf_x a(0,x)=a_0$. This gives
$$
a_0\int_{\xT^n}\la\nabla h\ra^2\dx\le \int_{\xT^n}B\dx\le \int_{\xT^n}1\dx=\la \xT^n\ra,
$$
which implies~\e{n2002}.

$ii)$ Here we use two simple tricks. Firstly, it follows from~\e{n1201} that
$$
a=1-B=\frac{1-G(h)h}{1+\la \nabla h\ra^2}.
$$
The positivity of the Rayleigh-Taylor coefficient $a=1-B$ (see Proposition~\ref{Coro:Zaremba-Taylor}) 
implies $G(h)h\le 1$ (pointwise). 
Let us prove that $G(h)h\ge-1$. Since 
$G(h)h=-h_t$, this is equivalent to the property that $h_t\le 1$, which holds here thanks to the assumption 
$h_t\arrowvert_{t=0}=-G(h_0)h_0\le 1$ and the maximum principle for the time derivative (see 
Theorem~\ref{NeatMax}). This proves that $\la G(h)h\ra\le 1$. 

The second simple trick is the identity
$$
B^2+\la V\ra^2=\frac{(G(h)h)^2+\la \nabla h\ra^2}{1+\la \nabla h\ra^2},
$$
which can be verified from~\e{n1201} by an elementary calculation. 
Then, $\la G(h)h\ra\le 1$ implies that
$B^2+\la V\ra^2\le 1$. 
This obviously implies that 
$\la V\ra^2\le 1$. One also deduces 
that $\la V\ra^2\le (1-B^2)=(1-B)(1+B)\le 2(1-B)$ since $B\le 1$ (as $1-B>0$). 
Using $(1-B)\nabla h=V$, 
we deduce from the two previous bounds for $V$ that 
$(1-B)^2\la \nabla h\ra^2\le \min\{1,2(1-B)\}$. We then divide by 
$(1-B)^2$ and the wanted 
inequality is a consequence of the lower bound $1-B=a\ge a_0$ (as seen above).
\end{proof}

\begin{prop}\label{T:Vp}
Assume that $n=1$ and consider an integer $p$ in $\{1\}\cup 2(\xN\setminus\{0\})$. 
Let $h$ be a regular solution to $\partial_th+G(h)h=0$ such that
\be\label{assuslope1}
\sup_{(t,x)\in [0,T]\times\xT} \la \partial_xh(t,x)\ra\le \sqrt{\frac{p}{3p-2}}.
\ee
Then, for all time $t$ in $[0,T]$, there holds
\be\label{n301}
\int_{\xT} \frac{V(t,x)^{2p}}{a(t,x)}\dx \le 
\int_{\xT} \frac{V(0,x)^{2p}}{a(0,x)}\dx.
\ee
\end{prop}

\begin{rema}
We shall prove a stronger result which includes a parabolic 
gain of regularity in $L^p$-spaces, see~\e{n141}. By Proposition \ref{prop:bornederiv}, assumption \eqref{assuslope1} can also be reduced to an assumption at time $t=0.$ 
\end{rema}

\subsection{Convexity inequalities}\label{S:convexintro}

We conclude this section by discussing additional identities which will be derived along the proofs.

In~\cite{CC-PNAS-2003,CC-CMP-2004},  C{\'o}rdoba and C{\'o}rdoba proved  
that, for any exponent $\alpha$ in $[0,1]$ and 
any $C^2$ function $f$ decaying sufficiently fast at infinity, 
one has the pointwise inequality
$$
2 f (-\Delta)^\alpha f \ge (-\Delta)^\alpha (f^2).
$$
This inequality has been generalized and applied to many different problems. 
To mention a few results, we quote the papers by Ju~\cite{Ju2005maximum}, 
Constantin and Ignatova~(\cite{ConstantinIgnatova1,ConstantinIgnatova2}), 
Constantin, Tarfulea and Vicol (\cite{ConstantinTV}), and we refer to the numerous references there in. 
Recently, C\'ordoba and Mart\'\i nez (\cite{CordobaM}) proved that
$$
\Phi'(f)G(h)f\ge G(h)\big(\Phi(f)\big)
$$
when $h$ is a $C^2$ function and 
$\Phi(f)=f^{2m}$ for some positive integer $m$. 
For our problems, 
we will need to apply this result for some functions $\Phi$ which are not powers. 
To do so, we will extend the previous result to 
the general case where $\Phi$ is a convex function and $h$ is $C^s$ for some $s>1$.

\begin{prop}\label{P:convexityintro}
Let $s>1$ and consider two functions $f,h$ in $C^s(\xT^n)$. For any $C^2$ convex function $\Phi\colon \xR\to\xR$, it  holds the pointwise inequality
\be\label{convexbis}
\Phi'(f)G(h)f\ge G(h)\big(\Phi(f)\big).
\ee
\end{prop}

In particular, for any function $f$, 
one has $2fG(h)f\ge G(h)(f^2)$ and hence the coefficient $\gamma$ defined by \eqref{eq:defgamma} in the equations for $B,V$ satisfies:
$$
\gamma\le 0.
$$
Now, to obtain the $L^p$-estimate for the inverse of $a=1-B$, we begin by computing that 
the function $\zeta\defn a^{-1}$ solves
$$
\partial_t\zeta^p-V\cdot \nabla \zeta^p-p \zeta^{p} G(h)a-p\gamma \zeta^{p+1}=0,
$$
and then we integrate over $\xT^n$. Since we want to prove that the integral of $\zeta^p$ decays and since $\gamma\le 0$, the contribution of the last term has a favorable sign. We then observe that the convexity inequality~\e{convexbis}, applied 
with $\Phi(r)=r^{-p+1}$, implies that 
$$
-\int\zeta^{p} G(h)a\dx\ge 0.
$$
So to complete the proof, it remains only to relate the integral of $V\cdot \nabla \zeta^p$ and 
the one of $\zeta^{p} G(h)a$. To do so, we integrate by parts to make appear the integral of 
$-\cnx(V) \zeta^p$. Then the desired decay estimate for the $L^p$-norm of $\zeta$  follows from 
the identity (see \S\ref{S:4})
$$
G(h)a=\cnx V.
$$
The maximum principle for $a$ then easily follows from the property that the infimum of $a$ 
is the supremum of $1/a$, which is the limit of its $L^p$-norms when $p$ goes to $+\infty$. 

The proof of Theorem~\ref{T:Vp} is quite delicate. 
We begin by establishing the following conservation law:
\be\label{n135}
\fract \int_{\xT} \frac{V^{2p}}{a}\dx+2p \int_{\xT} V^{2p-1} G(h)V\dx+(2p-1)\int_{\xT}\frac{\gamma  V^{2p}}{a^2}\dx=0
\ee
(here the space dimension is $n=1$).  
As in~\e{n212}, the inequality~\e{convexbis} implies that
\be\label{n136}
2p\int_{\xT}  V^{2p-1} G(h)V\dx \ge  \int_{\xT}V^p G(h)V^p\dx. 
\ee
Compared to the proof of Theorem~\ref{T:Bp}, 
the main difficulty is that the contribution of the term coming from $\gamma$ has 
not a favorable sign. Indeed, since $\gamma\le 0$, one has
$$
\int_{\xT} \frac{\gamma  V^{2p}}{a^2}\dx\le 0,
$$
so that one cannot deduce the wanted decay estimate~\e{n301} from \e{n135} and \e{n136}. 
To overcome this difficulty, we shall prove that the positive contribution \e{n136} dominates. To do so, we need 
a new identity relating $\gamma$ and~$V$. 
This is where we need to restrict the problem to space dimension $n=1$.
Indeed, if $n=1$, then one can exploit the fact that
$$
\Delta_{x,y}\big((\partial_x\phi)^2-(\partial_y\phi)^2\big)=0
$$
for any harmonic function $\phi$, to obtain $G(h)(B^2)-2BG(h)B=G(h)(V^2)-2VG(h)V$ which gives that
$$
\gamma=\frac{2}{1+(\partial_xh)^2}\Big(G(h)(V^2)-2V G(h)V\Big).
$$
The assumption~\e{assuslope1} then allows to absorb the contribution of $VG(h)V$ by the 
parabolic gain of regularity~\e{n136}. On the other hand, 
the convexity inequality~\e{convexbis} and the positivity of some coefficient  
imply that the contribution of $G(h)(V^2)$ has a favorable sign, giving some extra parabolic regularity. 
Then we conclude the proof using again  
the identity $G(h)B=-\cnx V$. 

\subsection{Organisation of the paper}

We begin in Section~\ref{S:DN} by recalling various results for the Dirichlet to Neumann operator. 
Then in Section~\ref{S:max} we recall the Zaremba principle and apply this result 
to prove that: $i)$ the Taylor coefficient $1-B$ is always positive (this is a classical result), 
$ii)$ the comparaison principle $h_1\le h_2$ stated in Proposition~\ref{prop:orderingintro}, 
$iii)$ the convexity inequality $\Phi'(f)G(h)f\ge G(h)\big(\Phi(f)\big)$ of Proposition~\ref{P:convexityintro}. 

The identities for $B$ and $V$ stated in Proposition~\ref{Prop:eqBV} are proved in Section~\ref{S:eqBV}. 
In the same section, we use 
these identities to prove Proposition~\ref{P:energyconvex} 
(see \S\ref{S:energyconvex}). 

The sharp maximum principle for all derivatives is proved in Section~\ref{S:sharpmax}. 
Then Theorem~\ref{T:Bp} is 
proved in Section~\ref{S:Lp1} and Proposition~\ref{T:Vp} in Section~\ref{S:Lp2}. 

The Cauchy problem is studied in Section~\ref{S:Cauchy}.

\section{The Dirichlet to Neumann operator}\label{S:DN}

We gather in this section some results about the Dirichlet to Neumann operator in domains 
with H\"older regularity. 

For $s\in ]0,+\infty[\setminus \xN$, we denote by 
$C^s$ the space of bounded functions whose derivatives of order $[s]$  are uniformly H\"older continuous with 
exponent $s-[s]$. 
\begin{prop}\label{P:2.2}
Consider two numbers $s,\sigma$ such that
$$
0 <\sigma \le s,\quad 1 < s\le +\infty, \quad \sigma\not \in \xN,\quad s\not \in  \xN.
$$
Let $h\in C^s(\xT^n)$ and introduce the domain
$$
\Omega=\{ (x,y)\in \xT^n\times \xR\,:\, y<h(x)\}.
$$
For any function $\psi\in C^\sigma(\xT^n)$, there exists a unique function 
$\phi\in C^\infty(\Omega)\cap C^\sigma(\overline{\Omega})$ such that~$\nabla_{x,y}\phi$ belongs to $L^2(\Omega)$ and
\begin{equation}\label{n2.2}
\left\{
\begin{aligned}
&\Delta_{x,y}\phi=0\quad\text{in }\Omega,\\
&\phi(x,h(x))=\psi(x).
\end{aligned}
\right.
\end{equation}
\end{prop}
\begin{proof}
This is classical when $s=+\infty$ (which is the only case required to justify the computations 
in this paper). 
\end{proof}

In the sequel we shall call the unique such $\phi$ the variational solution.

\begin{coro}\label{estiDN}
Consider two numbers $s,\sigma$ such that
\be\label{a:2.2}
1 <\sigma \le s\le +\infty, \quad \sigma\not \in \xN,\quad s\not \in \xN.
\ee
If $\psi \in C^\sigma(\xR^d)$ and $h\in C^{s} ( \xR ^d)$, then
$$
G(h)\psi\in C^{\sigma-1}(\xR^d).
$$
\end{coro}
\begin{proof}
Since
$$
G(h)\psi=(\partial_y \phi -\nabla h\cdot \nabla \phi)\arrowvert_{y=h},
$$
this result follows from Proposition~\ref{P:2.2}. Indeed, since $\nabla \phi$ belongs to $C^{\sigma-1}(\Omega)$ with $\sigma-1>0$ 
by assumption~\e{a:2.2}, 
one can take the trace on the boundary $\{y=h\}$.
\end{proof}

The expression $G(h)\psi$ is linear in $\psi$ but depends nonlinearly in $h$. 
This is the main difficulty to study the Hele-Shaw equation. The following result 
helps to understand the dependence in $h$.
\begin{prop}\label{P:shape}
Consider two real numbers $s,\sigma$ such that
$$
1 <\sigma \le s\le +\infty, \quad \sigma\not \in \xN,\quad s\not \in \xN.
$$
Let $\psi \in C^\sigma(\xT^n)$ and $h\in C^{s} (\xT^n)$. 
Then there is a neighborhood $\mathcal{U}_h\subset C^{s}(\xT^n)$ of $h$ such that
the mapping
$$
\eta \in \mathcal{U}_h \mapsto G(\eta)\psi \in C^{\sigma-1}(\xT^n)
$$
is differentiable. Moreover, for all $\zeta\in C^{s}(\xT^n)$, we have 
\be\label{n:shape}
d G(h)\psi \cdot  \zeta \defn
\lim_{\eps\rightarrow 0} \frac{1}{\eps}\big\{ G(h+\eps \zeta)\psi -G(h)\psi\big\}
= -G(h)(\mathfrak{B}\zeta) -\cnx (\mathfrak{V}\zeta),
\ee
where
$$
\mathfrak{B}=\frac{G(h)\psi+\partialx h\cdot\partialx\psi}{1+\la \partialx h\ra^2},
\quad \mathfrak{V} =\partialx\psi-\mathfrak{B}\partialx h.
$$
\end{prop}
\begin{proof}
This is proved by Lannes (see~\cite{LannesJAMS}) when 
the functions are smooth, which is the only case required to justify the computations 
in this article.
\end{proof}

\section{Maximum principles}\label{S:max}

In this section, we discuss several applications of Zaremba's principle. 
We begin by recalling the classical maximum principle.

\begin{prop}
Let $h\in C^1(\xT^n)$ and set
$$
\Omega=\{ (x,y)\in \xT^n\times \xR\,:\, y<h(x)\}.
$$
Consider a function $u\in C^2(\Omega)\cap C^0(\overline{\Omega})$ with $\nabla_{x,y}u\in L^2(\Omega)$ 
satisfying
$$
-\Delta u\ge 0\quad\text{in }\Omega,\qquad 
u\ge0 \quad \text{on }\partial\Omega.
$$
Then
$$
u\ge 0.
$$ 
\end{prop}
The original version of the Zaremba principle (see~\cite{Zaremba}) states that, if 
$\partial \Omega$ is $C^2$ and $x_0\in \partial \Omega$, then
$$
-\Delta u=f\ge 0 \quad\text{in }\Omega, \quad 
u(x)>u(x_0) \quad \text{in }\Omega \quad 
\Rightarrow \quad \partial_n u(x_0)<0. 
$$
We shall use a version which holds in domain which are less regular (see Safonov~\cite{Safonov}, 
Apushkinskaya-Nazarov~\cite{Apushkinskaya-Nazarov} and 
Nazarov~\cite{Nazarov-SIAM}).

\begin{theo}
Let $h\in C^s(\xT^n)$ with $s\in (1,+\infty)$ and set
$$
\Omega=\{ (x,y)\in \xT^n\times \xR\,:\, y<h(x)\}.
$$
Consider a function
$$
u \in C^2(\Omega)\cap C^1(\overline{\Omega}),
$$
satisfying $\Delta u\le 0$. 
If $u$ attains its minimum at a point $x_0$ of 
the boundary, then
$$
\partial_n u(x_0)<0.
$$ 
\end{theo}

In this section, we shall work out three applications of this argument. 

\subsection{Positivity of the Rayleigh--Taylor coefficient}

Our first application of Zaremba's principle is not new: 
we prove that the Rayleigh--Taylor stability condition is satisfied 
(see \cite{CCFGLF-Annals-2012,CCFG-ARMA-2013,CCFG-ARMA-2016,Cheng-Belinchon-Shkoller-AdvMath}).

\begin{prop}\label{Coro:Zaremba-Taylor}
Let $h\in C^s(\xT^n)$ with $s\in (1,+\infty)$ and set
$$
\Omega=\{ (x,y)\in \xT^n\times \xR\,:\, y<h(x)\}.
$$
Define $\phi$ as the variational solution to
\begin{equation*}
\Delta_{x,y}\phi=0\quad\text{in }\Omega,\quad \phi(x,h(x))=h(x).
\end{equation*}
and set $B= (\partial_y \phi)\ah$. 
Then
$$
1-B>0.
$$
\end{prop}
\begin{proof}
As already mentioned this result is not new when the free surface is smoother. 
We repeat here a classical proof in the water-waves theory (see~\cite{WuInvent,LannesJAMS}), in order 
to carefully check that the result remains valid when the boundary is only $C^s$ with $s>1$.

Given $\ell>0$, set
$$
\Omega_\ell=\left\{ (x,y)\in\xT^n\times\xR\,;\, -\ell<y<h(x)\right\},
$$
and introduce
$$
P=\phi-y.
$$
Then $P$ is an harmonic function in $\Omega_\ell$ vanishing on $\Sigma\defn \{y=h(x)\}$. 
Moreover, since~$\partial_y\phi$ goes to $0$ when $y$ goes to $-\infty$, one gets that, if~$\ell$ is large enough, then
$$
\partial_yP\arrowvert_{y=-\ell}<0.
$$
Since $\partial_nP=-\partial_y P$ on $\{y=-\ell\}$, one infers from the Zaremba principle that $P$ cannot reach its minimum on $\{y=-\ell\}$. So $P$ reaches its minimum on $\Sigma$. On the other hand, $P$ is constant on $\Sigma$. This shows that 
$P$ reaches its minimum on any point of $\Sigma$. Using again the Zaremba principle, one concludes that 
$\partial_n P<0$ on any point of $\Sigma$. So, to conclude the proof, it remains only to relate 
$\partial_n P$ and $\partial_y P$ on~$\Sigma$. To do so, we apply the chain rule to the equation $P(x,h(x))=0$. This gives
$$
(\nabla P)\ah=-(\partial_yP)\ah\nabla h.
$$
Recalling that $n=(1+|\nabla h|^2)^{-1/2}\left(\begin{smallmatrix} -\nabla h\\ 1\end{smallmatrix}\right)$, 
and using the previous identity, one has
$$
(\partial_n P)\ah=\frac{1}{\sqrt{1+|\nabla h|^2}}(\partial_y P-\nabla h\cdot \nabla P)\ah=\sqrt{1+|\nabla h|^2}(\partial_y P)\ah.
$$
This proves that $(\partial_y P)\ah<0$ on $\Sigma$, which means that $B-1=(\partial_y \phi)\ah-1<0$, which is the desired  inequality.
\end{proof}

\subsection{A comparison principle}
Our second application of the Zaremba principle gives a comparison principle for solutions of the Hele-Shaw equation.

\begin{prop}\label{prop:ordering}
Let $h_1,h_2 \in C^\infty([0,T]\times \xT^n)$ be two solutions of 
the Hele-Shaw 
equation $\partial_t h + G(h)h=0$ such that 
$h_1(0,\cdot) \leq h_2(0,\cdot)$. Then 
$$
h_1(t,\cdot) \leq h_2(t,\cdot)
$$
for all $t \in [0,T]$.
\end{prop}
\begin{proof}
	Define the set 
	$$
	\mathcal{T} := \Big\{ t \in [0,T] \, : \, h_1(t,\cdot) \leq h_2(t,\cdot)\Big\},  
	$$
	so that the statement of Proposition \ref{prop:ordering} reduces to the fact that $\mathcal{T}=[0,T].$ 
	
	We claim that whenever $t_* \in \mathcal{T}$, $\mathcal{T}$ contains an open
	neighborhood of $t_*$ in $[t_*,T].$ Since $0$ belongs to $\mathcal{T}$ by assumption, and since  
	$\mathcal{T}$ is a closed subset of $[0,T]$ by continuity of $h_1$ and $h_2,$ the proof 
	of Proposition \ref{prop:ordering} will follow from the claim. 
	
	For the later, we distinguish three cases. 

        {\it Case i): $h_1(t_*,x) < h_2(t_*,x)$, for all $x \in \mathbb{T}^n.$} This is the easiest
	case: by compactness and continuity it follows that the same inequality holds for all $t$ in
	an open neighborhood of $t_*$ in $[t_*,T]$ (and actually also in $[0,T]$).

	{\it Case ii): $h_1(t_*,\cdot)=h_2(t_*,\cdot).$}  By local well-posedness of
	the Hele-Shaw equation (see Theorem~\ref{T:Cauchy}), it follows that $h_1(t,\cdot)=h_2(t,\cdot)$ for all $t \in [t_*,T],$
	and in particular $[t_*,T] \subseteq \mathcal{T}.$

	{\it Case iii): None of the latter.} In that situation, the set 
	$C_*:=\{x \in \xT^n \, : \, h_1(t_*,x)=h_2(t_*,x)\}$ is a non empty proper 
	subset of $\xT^n.$ Consider an arbitrary element $x_* \in C_*$, and for notational
	convenience set $z_*:=(x_*,h_1(t_*,x_*))=(x_*,h_2(t_*,x_*)).$  
	The function $h_2(t_*,\cdot)-h_1(t_*,\cdot)$ being non negative on $\xT^n$ and
	vanishing at the point $x_*$, we deduce that $Dh := \nabla_x h_1(t_*,x_*) = \nabla_x h_2(t_*,x_*)$ 
	and also that $n:= n_1 = n_2$, where for $i=1,2$ we denoted by $n_i$ the outward unit normal 
	to $\Omega_i$ at the point $z_*$. For $i=1,2,$ let $P_i$ be the unique harmonic function defined in
	$\Omega_i := \{ (x,y) \in \xT^n \times \xR \, : \, y<h_i(t_*,x)\}$ 
	and which satisfies $P_i = 0$ on
	$\partial \Omega_i$ and $P_i+y$ is bounded on $\Omega_i.$ From the maximum
	principle we infer that $P_i$ is positive in $\Omega_i,$ and in particular since $t_*\in
	\mathcal{T}$ and since $C_*$ is proper it follows that $P_2$ is non negative and not identically 
	zero on $\partial \Omega_1.$ A further application of the maximum principle yields that  
	$P_2 > P_1$ on $\Omega_1$, and from Zaremba's principle it then follows that 
	\begin{equation}\label{eq:p1p2zaremba}
		\partial_n \left(P_2-P_1\right)(z_*)<0.
	\end{equation}
	On the other hand, subtracting the Hele-Shaw equations satisfied by $h_2$ and $h_1$ we
	obtain 
	\[
		\partial_t (h_2-h_1)(t_*,x_*) = -\sqrt{1+|Dh|^2} \partial_n
		\left(P_2-P_1\right)(z_*),
	\]
	and therefore by \eqref{eq:p1p2zaremba} this implies   
	\begin{equation}\label{eq:vitesseordo}
	\partial_t h_2(t_*,x_*) > \partial_t h_1(t_*,x_*). 
	\end{equation}
	By compactness of $C_*$, and the fact that $\partial_t h_1$ and $\partial_t h_2$
	are continuous functions, we derive the existence of $\eps>0$ such that 
	$\partial_t h_2(t_*,x) > \partial_t h_1(t_*,x) + \eps$ for all $x$ in some open 
	neighborhood $\mathcal{O}$ of $C_*$ in $\mathbb{T}^n.$  On the other hand, on the compact set
	$\xT^n \setminus \mathcal{O}$, the function $h_2(t_*,\cdot) - h_1(t_*,\cdot)$ is
	positive and therefore bounded from below by some positive constant. By elementary real
	analysis, this also implies that $\mathcal{T}$ contains an open neighborhood of $t_*$ in
	$[t_*,T].$
\end{proof}

\subsection{A convexity inequality for the Dirichlet to Neumann operator}

As explained in the introduction, our third application of the Zaremba  
principle is a convexity inequality which we believe is of independent interest.

\begin{prop}
Let $s\in (1,+\infty)$ and consider two functions $f,h$ in $C^s(\xT^n)$. For any $C^2$ convex function $\Phi\colon \xR\to\xR$, it  holds the pointwise inequality
\be\label{convex}
\Phi'(f)G(h)f\ge G(h)\big(\Phi(f)\big).
\ee
\end{prop}
\begin{rema}
We consider only periodic functions but the proof is extremely simple and easy to adapt to other settings.
\end{rema}
\begin{proof}
Denote by $\zeta$ (resp.\ $\xi$) the harmonic extension of 
$f$ (resp.\ $\Phi(f)$), so that
\begin{align*}
&\Delta_{x,y}\zeta=0\quad\text{in }\Omega,\quad \zeta\ah=f,\\
&\Delta_{x,y}\xi=0\quad\text{in }\Omega,\quad \xi\ah=\Phi(f).
\end{align*}
By assumption, $\zeta$ and $\xi$ belong to $C^2(\Omega)\cap C^1(\overline{\Omega})$. 
By definition of the Dirichlet to Neumann operator and using the chain rule, one has
$$
G(h)\big(\Phi(f)\big)-\Phi'(f)G(h)f= \sqrt{1+|\nabla h|^2}\partial_n \big(\xi-\Phi(\zeta)\big).
$$
It suffices then to prove that the difference $u=\xi-\Phi(\zeta)$ satisfies $\partial_n u\le 0$ on 
$\partial\Omega$. To do so, using that $\Phi$ is convex, we observe that
$$
\Delta \Phi(\zeta)=\Phi'(\zeta)\Delta \zeta+\Phi''(\zeta)|\nabla \zeta|^2=\Phi''(\zeta)|\nabla \zeta|^2\ge 0.
$$
Thus, we deduce that
$$
-\Delta_{x,y}u\ge 0\quad\text{in }\Omega,\quad u=0 \quad\text{on }\partial\Omega.
$$
It follows from the maximum principle that $u\ge 0$ in $\Omega$. 
Since $u$ vanishes on 
$\partial\Omega$, we infer that
$$
\forall x\in \partial\Omega,~\forall t>0,\quad u(x-tn)\ge u(x),
$$
where $n$ is the outward unit normal to the boundary. Since $u$ belongs to $C^1(\overline{\Omega})$, this immediately implies that 
$\partial_n u\le 0$, which completes the proof.
\end{proof}

There are several applications that one could work out of this convexity inequality. 
We begin by proving the version of the maximum principle for the Hele-Shaw equation stated in~\S\ref{S:maxintro}.

\begin{prop}
Let $n\ge 1$ and consider an integer $p$ in $\{1\}\cup 2\xN$. 
Assume that $h\in C^\infty([0,T]\times \xT^n)$ is a solution to $\partial_th+G(h)h=0$. 
Then, for all time $t$ in $[0,T]$, there holds
\be\label{n210v2}
\int_{\xT^n} h(t,x)^{2p}\dx
+2\int_0^t\int_{\xT^n}h^p G(h)(h^p)\dx\dt 
\le \int_{\xT^n} h(0,x)^{2p}\dx.
\ee
Consequently, for all time $t$ in $[0,T]$,
\be\label{n211v2}
\sup_{x\in \xT^n}\la h(t,x)\ra\le \sup_{x\in \xT^n}\la h(0,x)\ra.
\ee
\end{prop}

\begin{proof}
Since
$$
\fract\int_{\xT^n} h^{2p}\dx+2p\int_{\xT^n}h^{2p-1}G(h)h\dx=0,
$$
the decay estimate~\e{n210v2} will be proved if we justify that
\be\label{n212}
2p\int_{\xT^n}h^{2p-1}G(h)h\dx\ge 2\int_{\xT^n}h^p G(h)h^p\dx.
\ee
To do so, write $h^{2p-1}G(h)h=h^{p} \big( h^{p-1}G(h)h\big)$ and then 
use the inequality~\e{convex} applied with $\Phi(f)=f^p$ (the function $\Phi$ is convex since $p\in \{1\}\cup 2\xN$): this gives that $p h^{p-1}G(h)h\ge G(h)h^p$. We thus have proved~\e{n212} and hence \e{n210v2}. 
Now, we claim that $\int h^p G(h)(h^p)\dx\ge 0$. Indeed, 
we have $\int \psi G(h)\psi\dx \ge 0$ for any function $\psi$ since 
\be\label{n217}
\int \psi G(h)\psi\dx=\iint_\Omega \la \nabla_{x,y}\varphi\ra^2\dy\dx
\ee
where $\varphi$ is the harmonic extension of $\psi$ (see~\e{defi:varphipsi}). 
This implies that the $L^{2p}$-norm of $h$ decays. 
Then we deduce \e{n211v2}~by arguing that the $L^\infty$-norm of $h$ is the limit of its $L^{2p}$-norms when 
$p$ goes to $+\infty$ (see the end of the proof of Theorem~\ref{T:Bpbis} for details).
\end{proof}

\section{Evolution equations for the derivatives}\label{S:eqBV}
We now consider the evolution equation $\partial_t h+G(h)h=0$. 
We denote by $\phi(t,x,y)$ the unique solution to 
\begin{equation}\label{n401}
\left\{
\begin{aligned}
&\Delta_{x,y}\phi=0\quad\text{in }\{(x,y)\in\xT^n\times\xR\,;\, y<h(t,x)\},\\
&\phi(t,x,h(t,x))=h(t,x),
\end{aligned}
\right.
\end{equation}
and we use the notations
\be\label{n402}
B(t,x)= (\partial_y \phi)(t,x,h(t,x)),\quad 
V(t,x) = (\nabla_x \phi)(t,x,h(t,x)).
\ee
In this section we derive two key evolution equations for $B$ and $V$.

\subsection{Some known identities}\label{S:4}

We begin by recalling some key identities relating $B$, $V$ and~$h$.

\begin{prop}\label{cancellation}

$i)$ 
The functions $B$ and $V$ are given in terms of $h$ by means of the formula 
\begin{equation}\label{defi:BV}
B= \frac{G(h)h+\la \partialx h\ra^2}{1+|\partialx  h|^2},
\qquad
V=(1-B)\partialx h.
\end{equation}

$ii)$ $B$ and $V$ are related by
\be\label{I:G(h)B}
G(h)B=-\cnx V.
\ee

$iii)$ Moreover, for any integer $1\le i\le n$, there holds
\be\label{n3.5}
\partial_i B -G(h)V_i=(\partial_i h)G(h)B+\sum_{j} (\partial_j h)(\partial_i V_j).
\ee
In addition, if $n=1$ then $\partial_x B=G(h)V$.
\end{prop}
\begin{proof}
These results are not new. Indeed, these identities play a crucial 
role in the water-wave theory (see~\cite{ABZ3,BLS,LannesJAMS} for \e{defi:BV}--\e{I:G(h)B} and \cite{ABZ3} for \e{n3.5}). 
We recall the proof of this proposition for the sake of completeness. 

In this proof the time variable is seen as a parameter and we skip it.

$i)$ The chain rule implies that
$$
\nabla h=\nabla (\phi(x,h(x)))=(\nabla \phi +(\partial_y \phi)\nabla h)\ah=V+B\nabla h
$$
which implies that $V=(1-B)\nabla h$. On the other hand, by definition of the operator $G(h)$, one has
$$
G(h)h=\big( \partial_y \phi-\nabla h\cdot \nabla \phi\big)\ah=B-V\cdot \nabla h,
$$
so the identity for $B$ in \e{defi:BV} follows from $V=(1-B)\nabla h$.

$ii)$ By definition, one has
\begin{equation*}
B=(\partial_y \phi)\ah.
\end{equation*}
Therefore the function $\Phi$ defined by $\Phi (x,y)=\partial_y \phi (x,y)$  satisfies
$$
\Deltayx \Phi =0,\quad \Phi\arrowvert_{y=h}=B.
$$
Directly from the definition of the Dirichlet to Neumann operator, we  have
$$
G(h)B=
\partial_y \Phi - \partialx h\cdot \partialx \Phi \big\arrowvert_{y=h}.
$$
So it suffices to show that $\partial_y \Phi - \partialx h\cdot \partialx \Phi \big\arrowvert_{y=h}=-\cnx V$. 
To do that we first write that $\partial_y \Phi=\partial_y^2 \phi=-\Delta\phi$ to obtain 
$$
\big(\partial_y \Phi - \partialx h\cdot \partialx \Phi \big)\big\arrowvert_{y=h}=
-\big(\Delta \phi + \partialx h\cdot \partialx \partial_y\phi\big) \big\arrowvert_{y=h},
$$
which implies the desired result by using the chain rule:
$$
\big(\Delta \phi +\partialx h\cdot \partialx \partial_y\phi \big)\big\arrowvert_{y=h}=\cnx (\nabla\phi \big\arrowvert_{y=h}).
$$
This proves statement $ii)$.

$iii)$ Directly from the definitions of~$B$ and~$V$ 
($B=\partial_y\phi\arrowvert_{y=h}$,~$V=\nabla\phi\arrowvert_{y=h}$), and using the chain rule, 
we compute that
\begin{align*}
\partial_i B - \sum_{j=1}^d \partial_i V_j \partial_jh
 &= \big[\partial_i \partial_y \phi +\partial_i h \partial_y^2\phi \big] \big\arrowvert_{y=h} 
 -\sum_{j=1}^d 
 \partial_j h \big[\partial_i \partial_j \phi 
 +\partial_i h \partial_j\partial_y\phi \big] \big\arrowvert_{y=h} \\
&= \big[\partial_y \partial_i  \phi 
- \sum_{j=1}^d \partial_j h \partial_i \partial_j \phi\big] \big\arrowvert_{y=h} 
+\partial_i h \big[ \partial_y^2 \phi -\sum_{j=1}^d 
\partial_j h \partial_j\partial_y\phi \big] \big\arrowvert_{y=h} .
\end{align*}
Let $1\le i\le n$. Notice that $\theta_i=\partial_i \phi$ solves 
$$
\Delta_{x,y}\theta_i =0 \text{ in }\Omega, \quad \theta_i \arrowvert_{y=h} = V_i.
$$
Then
$$
G(h)V_i = (\partial_y\theta_i - \nabla h \cdot \nabla\theta_i)\arrowvert_{y=h}.
$$
Similarly, as already seen, one has
$$
\big[ \partial_y^2 \phi -\sum_{j=1}^d 
\partial_j h \partial_j\partial_y\phi \big] \big\arrowvert_{y=h} =G(h)B.
$$
We thus have proved that
$$
\partial_i B - \sum_{j=1}^d \partial_i V_j \partial_jh=G(h)V_i+(\partial_i h)G(h)B,
$$
which completes the proof of \e{n3.5}. Now notice that, if $n=1$, 
then \e{n3.5} reduces to
$$
\partial_x B -G(h)V=(\partial_x h)G(h)B+ (\partial_x h)(\partial_x V),
$$
which yields $\partial_x B-G(h)V=0$ since $G(h)B=-\partial_x V$. This 
completes the proof of statement~$iii)$.
\end{proof}

\subsection{Parabolic equations}
We are now in position to derive the parabolic evolution equations for $B$ and $V$. We begin by studying $B$.

\begin{prop}
Assume that $h\in C^\infty([0,T]\times \xT^n)$ satisfies $\partial_t h+G(h)h=0$ and define 
$B,V$ by \e{defi:BV}. Then $B$ and $V$ belong to $C^\infty([0,T]\times \xT^n)$. Moreover, $B$ satisfies
\be\label{defi:gamma}
\partial_t B-V\cdot \nabla B +(1-B)G(h)B=\gamma,
\ee
where
\be\label{ngammadef}
\gamma=\frac{1}{1+|\nabla h|^2}\Big(G(h)\big(B^2+|V|^2\big)-2BG(h)B-2V\cdot G(h)V\Big).
\ee
Moreover, the coefficient $\gamma$ satisfies
\be\label{ngammasign}
\gamma\le 0.
\ee
\end{prop}
\begin{proof}
Assuming~\e{ngammadef}, the fact that $\gamma$ is negative follows from the convexity inequality~\e{convex}. Indeed, 
this inequality implies that
$$
G(h)(B^2)\le 2 B G(h)B, \quad G(h)(V_j^2)\le 2V_j G(h)V_j \quad (1\le j\le n),
$$
which implies the desired inequality~\e{ngammasign}.

It remains to obtain the identity~\e{ngammadef}. Since
$$
B=\frac{G(h)h+\la \nabla h\ra^2}{1+\la \nabla h\ra^2},
$$
the fact that $B$ belongs to $C^\infty([0,T]\times \xT^n)$ follows 
from the properties of the Dirichlet to Neumann operator recalled in Section~\ref{S:DN}. To obtain~\e{defi:gamma}, we first notice that, 
for any derivative $\partial\in \{\partial_t,\partial_{1},\ldots,\partial_{n}\}$ where $\partial_j=\partial_{x_j}$, 
one has
$$
\partial B=\frac{1}{1+\la \nabla h\ra^2}\left( \partial \big(G(h)h\big)+\left(1-\frac{G(h)h+\la \nabla h\ra^2}{1+\la \nabla h\ra^2}\right)\partial \la\nabla h\ra^2\right).
$$
This yields
\be\label{nBA1A2}
\partial_t B-V\cdot \nabla B=\frac{1}{1+\la \nabla h\ra^2}\left(A_1+A_2\right),
\ee
where
$$
A_1= (\partial_t -V\cdot \nabla)G(h)h,\quad A_2=2(1-B)\nabla h\cdot \big(\nabla \partial_t h -V\cdot \nabla \nabla h\big).
$$
We begin by computing the term $A_2$. 
To do so, we use
$$
(1-B)\nabla h=V
$$
to write $A_2=2V\cdot \big(\nabla \partial_t h -V\cdot \nabla \nabla h\big)$ and hence
$$
A_2=2V\cdot \nabla( \partial_t h -V\cdot \nabla h)+2\sum_{j,k} V_k(\partial_k V_j)\partial_j h.
$$
Since
$$
\partial_th =-G(h)h=-B+V\cdot \nabla h,
$$
this gives
\be\label{nA2}
A_2=-2V\cdot \nabla B+2\sum_{j,k} V_k(\partial_k V_j)\partial_j h.
\ee

We now move to $A_1$. We shall exploit the shape derivative formula~\e{n:shape}. This formula implies that
$$
\partial_t G(h)h=G(h)\big((1-B)\partial_t h\big)-\cnx (V \partial_{t}h),
$$
and similarly
$$
\partial_{j}G(h)h=G(h)\big((1-B)\partial_{j} h\big)-\cnx (V \partial_{j}h).
$$
Recall also that
$$
(1-B)\partial_{j} h=V_j
$$
and notice that
$$
\sum_{j}V_j \cnx (V \partial_{j}h)=\cnx \big( (V\cdot \nabla h)V\big)-\sum_{j,k}(\partial_k V_j)V_k (\partial_j h).
$$
By combining the previous observations, we get that
\begin{align*}
A_1&=G(h)\big((1-B)\partial_t h\big)-V\cdot  G(h)V\\
&\quad -\cnx((\partial_{t}h-V\cdot \nabla h)V)-\sum_{j,k}(\partial_k V_j)V_k (\partial_j h).
\end{align*}
As already mentioned, one has $\partial_{t}h=-B+V\cdot \nabla h$, so that
$$
\cnx((\partial_{t}h-V\cdot \nabla h)V)=-\cnx (B V)=-B\cnx V-V\cdot \nabla B
$$
and for the same reason, 
\begin{align*}
(1-B)\partial_t h&=(1-B)(-B+V\cdot \nabla h)=-B+B^2+V\cdot \big( (1-B)\nabla h\big)\\
&=-B+B^2+\la V\ra^2,
\end{align*}
where we used $(1-B)\nabla h=V$ in the last identity. Consequently, we deduce that
\begin{align*}
A_1&=-G(h)B+G(h)(B^2+V^2)-V\cdot  G(h)V\\
&\quad +B\cnx V+V\cdot \nabla B-\sum_{j,k}(\partial_k V_j)V_k (\partial_j h).
\end{align*}

By combining this with \e{nA2} and simplifying the result, we have
\begin{align*}
A_1+A_2&=-G(h)B+G(h)(B^2+V^2)-V\cdot  G(h)V\\
&\quad +B\cnx V-V\cdot \nabla B+\sum_{j,k}(\partial_k V_j)V_k (\partial_j h).
\end{align*}

The key point is that one can further simplify this expression by means of Lemma~\ref{cancellation}, which implies that
$$
\sum_{j} (\partial_j h)(\partial_k V_j)=\partial_k B -G(h)V_k-(\partial_k h)G(h)B.
$$
Consequently,
$$
\sum_{j,k}(\partial_k V_j)V_k (\partial_j h)=V\cdot \nabla B-V\cdot G(h)V- (V\cdot \nabla h)G(h)B,
$$
and hence, since $V\cdot \nabla h=(1-B)\la \nabla h\ra^2$, we conclude that
$$
\sum_{j,k}(\partial_k V_j)V_k (\partial_j h)=V\cdot \nabla B-V\cdot G(h)V-(1-B)\la \nabla h\ra^2 G(h)B.
$$
As a result,
\begin{align*}
A_1+A_2&=-G(h)B+G(h)(B^2+V^2)-2V\cdot  G(h)V\\
&\quad +B\cnx V- (1-B)\la \nabla h\ra^2 G(h)B.
\end{align*}
Now we write
$$
G(h)B=(1-B)G(h)B+BG(h)B
$$
to obtain
\begin{align*}
A_1+A_2&=-\big(1+\la \nabla h\ra^2\big)(1-B)G(h)B\\
&\quad +G(h)(B^2+V^2)+B\cnx V-BG(h)B-2V\cdot  G(h)V.
\end{align*}
So the desired formula follows from the identity $\cnx V=-G(h)B$ and \e{nBA1A2}.
\end{proof}

\begin{prop}\label{P4.2}
Assume that $h\in C^\infty([0,T]\times \xT^n)$ satisfies $\partial_t h+G(h)h=0$. Then $V$ belongs to $C^\infty([0,T]\times \xT^n)$ and satisfies
\be\label{n6.3}
\partial_t V-V\cdot \nabla V +(1-B)G(h)V+\frac{\gamma}{1-B} V=0.
\ee
Furthermore, the unknown 
$$
Y=(1-B)\partial_t h
$$
satisfies
\be\label{n6.4}
\partial_t Y-V\cdot \nabla Y +(1-B)G(h)Y+\frac{\gamma}{1-B}Y=0.
\ee
\end{prop}
\begin{proof}
Let $1\le j\le n$ and set $h_j=\partial_j h$. Since
$$
\partial_j \big(G(h)h\big)=G(h)((1-B)h_j)-\cnx (V h_j),
$$
and since $V_j=(1-B)h_j$, we have
$$
\partial_t h_j+G(h)V_j -\cnx (V h_j)=0.
$$
Now we multiply this equation by $(1-B)$ and commute $(1-B)$ with $\partial_t$, to obtain, using again $V_j=(1-B)h_j$
$$
\partial_t V_j +(1-B)G(h)V_j +(\partial_t B)h_j-(1-B)\cnx( V h_j)=0.
$$
Now, write
$$
(1-B)\cnx( V h_j)=(1-B)V \cdot \nabla h_j+(1-B)(\cnx V)h_j
$$
and commute $(1-B)$ with $V\cdot \nabla$ to obtain
$$
(1-B)\cnx( V h_j)=V \cdot \nabla ((1-B)h_j)+\big((V\cdot\nabla B)+(1-B)(\cnx V)\big)h_j.
$$
Consequently, we have
$$
\partial_t V_j -V\cdot \nabla V_j
+(1-B)G(h)V_j +\big(\partial_t B-(V\cdot\nabla B)-(1-B)(\cnx V)\big)h_j=0.
$$
Recall that
$$
\cnx V=-G(h)B
$$
and
$$
\gamma=\partial_t B-(V\cdot\nabla B)+(1-B)G(h)B.
$$
This gives 
$$
\partial_t V_j-V\cdot \nabla V_j +(1-B)G(h)V_j+\gamma h_j=0,
$$
which is the desired result \e{n6.3}. 

The exact same arguments apply when $h_j$ is replaced by $\partial_t h$. 
Indeed, 
$$
\partial_t \big(G(h)h\big)=G(h)((1-B)(\partial_t h))-\cnx (V (\partial_t h)).
$$
Therefore, we obtain \e{n6.4} by repeating the previous computations.
\end{proof}

\subsection{A higher order energy}\label{S:energyconvex}

The aim of this paragraph is to prove Proposition~\ref{P:energyconvex} whose statement is recalled here. 
\begin{prop}
For any regular solution $h$, there holds
\begin{align}
&\mez\fract\int h^2\dx +\int h G(h)h \dx =0,\label{Energy}\\[1ex]
&\fract\int h G(h)h \dx+\int (1-B)(h_t^2+|\nabla h|^2)\dx=0.\label{Energy1/2}
\end{align}
\end{prop}
\begin{proof}
The first identity is the energy identity obtained by multiplying 
the Hele-Shaw equation by $h$. 
To prove the second one, we start from
$$
\partial_t G(h)h=G(h)\big((1-B)h_t\big)-\cnx(h_t V).
$$
Then
$$
\fract\int h G(h)h \dx=\int h_t G(h)h\dx 
+\int h G(h)((1-B)h_t)\dx -\int h \cnx(h_t V)\dx.
$$
Since $G(h)$ is self-adjoint, we have
$$
\int h G(h)((1-B)h_t)\dx  =\int (G(h)h) (1-B)h_t\dx.
$$
On the other hand, 
$$
G(h)h=-h_t,
$$
so
$$
\int h G(h)((1-B)h_t)\dx  =-\int (1-B)h_t^2\dx,
$$
hence,
$$
\fract\int h G(h)h \dx=-\int(1-B)h_t^2\dx 
+\int h_t G(h)h\dx  -\int h \cnx(h_t V)\dx.
$$
Integrating by parts the last term gives
$$
\fract \int hG(h)h\dx =-\int(1-B)h_t^2\dx 
+\int h_t \big(G(h)h +V\cdot\nabla h\big)\dx.
$$
Now, by definition of $G(h)h$, there holds
$$
G(h)h +V\cdot\nabla h=B,\quad h_t=-G(h)h=-B+V\cdot\nabla h
.
$$
As a result,
$$
\int h_t \big(G(h)h +V\cdot\nabla h\big)\dx=\int \big(-B^2+BV\cdot\nabla h\big)\dx.
$$
Next, we claim that we have the following elementary Rellich identity
$$
\int \big(|V|^2-B^2+2BV\cdot\nabla h\big)\dx=0.
$$
To see this, one verifies that $\Delta_{x,y}\phi=0$ implies that
$$
\partial_y(\phi_y^2- |\nabla_x\phi|^2)+2\cnx(\phi_y \nabla_x \phi)=0,
$$
and then applies the divergence theorem on $\Omega$ 
with the vector field 
$$
X=(-2\phi_y\nabla_x\phi, |\nabla_x\phi|^2-\phi_y^2).
$$ 

By combining the above results, we end up with
$$
\int h_t \big(G(h)h +V\cdot\nabla h\big)\dx=-\int \big(|V|^2+BV\cdot\nabla h\big)\dx.
$$
Since $V=(1-B)\nabla h$, this gives
$$
\int h_t \big(G(h)h +V\cdot\nabla h\big)\dx=-\int (1-B)|\nabla h|^2\dx,
$$
which concludes the proof.
\end{proof}

\section{Maximum principle for all the derivatives}\label{S:sharpmax}

We prove a maximum principle for all the spatial and time derivativves by adapting the Stampacchia's multiplier method. 
To do so, we begin by symmetrizing the equation. 

\subsection{Symmetrization of the equation}
\begin{prop}
Assume that $h\in C^\infty([0,T]\times \xT^n)$ satisfies $\partial_t h+G(h)h=0$. Recall the notation $a=1-B$ and introduce the operator $L(h)$ defined by
$$
L(h)f=-V\cdot \nabla f-\mez (\cnx V)f+\sqrt{a}\, G(h)\big(\sqrt{a}f\big).
$$
Set
$$
W=\sqrt{a}\nabla h, \quad Z=\sqrt{a} \partial_t h.
$$
Then
\begin{align}
&\Big(\partial_t +L(h) +\frac{\gamma}{2a} \Big)W=0, \label{n6.7}  \\ 
&\Big(\partial_t +L(h) +\frac{\gamma}{2a} \Big)Z=0,\label{n6.8} \\ 
&\Big(\partial_t +L(h) +\frac{\gamma}{2a} \Big)\sqrt{a}=0. \label{n6.9}
\end{align}
\end{prop}
\begin{rema}
Compared to the equation~\e{n6.3} for $V=a\nabla h$, the two improvements are that
$$
f\mapsto \sqrt{a}\, G(h)\big(\sqrt{a}f\big) \quad\text{ is self-adjoint}, 
$$
and
$$
f\mapsto -V\cdot \nabla f-\mez (\cnx V)f \quad \text{is skew-symmetric}.
$$
This is used later on to perform $L^2$-energy estimates.
\end{rema}
\begin{proof}
Since $V=\sqrt{1-B} W$, one has
\begin{align*}
\partial_t V-V\cdot \nabla V +\frac{\gamma}{1-B} V
&=\sqrt{1-B}\Big(\partial_t -V\cdot \nabla 
+\frac{\gamma}{1-B}\Big) W\\
&\quad -\frac{W}{2\sqrt{1-B}}\Big(\partial_t B  -V\cdot \nabla B \Big).
\end{align*}
Now, it follows from the equation~\e{defi:gamma} and the identity 
$G(h)B=-\cnx V$ (see~\e{I:G(h)B}) that
$$
\partial_t B-V\cdot \nabla B=\gamma+(1-B)\cnx V.
$$
As a result, it follows from the previous computations and the equation~\e{n6.3} for $V$ that
$$
\partial_t W-V\cdot \nabla W 
+\frac{\gamma}{1-B} W=-\sqrt{1-B}G(h)V+\frac{W}{2(1-B)}\Big(\gamma+(1-B)\cnx V\Big).
$$
We immediately obtain identity~\e{n6.7} for $W$ by simplifying this equation. One obtains the equation~\e{n6.8} (resp.\ \e{n6.9}) 
for~$Z$ (resp.\ $\sqrt{a}$) 
by repeating the same arguments starting from the equation~\e{n6.4} (resp.\ \e{defi:gamma}) for $Y$ (resp.\ $B$).
\end{proof}

\subsection{Application of the Stampacchia multiplier method}

We now prove Theorem~\ref{NeatMax} whose statement is recalled here.
\begin{theo}
Let $n\ge 1$. Consider a positive number $M>0$ and 
$h\in C^\infty([0,T]\times \xT^n)$ 
solution to $\partial_th+G(h)h=0$. 
Then for any derivative
$$
D\in \{\partial_t,\partial_{x_1},\ldots,\partial_{x_n}\},
$$
if, initially, $\sup_{x\in \xT^n} D h(0,x)\le M$, then, for all time $t$ in $[0,T]$,
\be\label{maxDh}
\sup_{x\in \xT^n} Dh(t,x)\le M.
\ee
\end{theo}
\begin{proof}
Consider a derivative 
$D$ in $\{\partial_t,\partial_{x_1},\ldots,\partial_{x_n}\}$ and set 
$U=\sqrt{a}Dh$.
It follows from \e{n6.7},~\e{n6.8} and \e{n6.9} that
\be\label{n6.19}
\Big(\partial_t +L(h) +\frac{\gamma}{2a} \Big)(U-M\sqrt{a})=0.
\ee
To obtain the bound~\e{maxDh}, 
we shall use Stampacchia's method. Introduce
$$
(U-M\sqrt{a})_+=\max\{ U-M\sqrt{a},0\}.
$$
The idea is that, since 
$$
U-M\sqrt{a}=\sqrt{a}(Dh-M),
$$
and since $\sqrt{a}>0$, to prove that $Dh\le M$ it is equivalent to prove that $(U-M\sqrt{a})_+=0$. 
To prove the latter result, we shall multiply the equation~\e{n6.19} 
by $(U-M\sqrt{a})_+$ and perform an $L^2$-energy estimate. 
To do so, we use the three following properties: $i)$ one has
$$
\int (U-M\sqrt{a})_+\partial_t (U-M\sqrt{a})\dx
=\mez\fract \int (U-M\sqrt{a})_+^2\dx,
$$
$ii)$ with $\mathcal{U}=\sqrt{a}(U-M\sqrt{a})=a(Dh-M)$ we have
$$
\int (U-M\sqrt{a})_+ L(h)(U-M\sqrt{a})\dx=\int \mathcal{U}_+G(h)\mathcal{U}\dx,
$$
and thirdly, the convexity inequality~\e{convex} implies that
$$
\int \mathcal{U}_+G(h)\mathcal{U}\dx\ge \int G(h)\Phi_+(\mathcal{U})\dx
$$
where $\Phi_+$ is the $W^{2,\infty}$ convex function whose derivative is $\Phi'_+(r)=\max\{0,r\}$. 
Since $\int G(h)\Phi_+(\mathcal{U})\dx=0$, as already seen, this proves that
$$
\int (U-M\sqrt{a})_+ L(h)(U-M\sqrt{a})\dx\ge 0.
$$
As a consequence, we deduce that
$$
\mez\fract \int (U-M\sqrt{a})_+^2\dx
+\int \frac{\gamma}{1-B}(U-M\sqrt{a})_+^2\dx\le 0.
$$
Consequently,
$$
y(t)=\int (U-M\sqrt{a})_+(t,x)^2\dx
$$
satisfies
$$
\dot{y}(t)\le C(t)y(t),
$$
with
$$
C(t)=\lA \frac{\gamma}{1-B}(t,\cdot) \rA_{L^\infty(\xT)}.
$$
Since $y(0)=0$ by assumption, the Gronwall's lemma implies that $y(t)=0$ for all time $t$, which terminates the proof.
\end{proof}

\section{Decay of the inverse of the Rayleigh--Taylor coefficient}\label{S:Lp1}

In this section we prove Theorem~\ref{T:Bp} whose statement is recalled here.

\begin{theo}\label{T:Bpbis}
Let $n\ge 1$ and consider an integer $p\in \xN$. 
Consider a regular solution $h$ of $\partial_th+G(h)h=0$ defined on $[0,T]$ 
and set $a=1-B$ where 
$B$ is as defined by~\e{n401}--\e{n402}.

$i)$ For all time $t$ in $[0,T]$, there holds
$$
\fract \int_{\xT^n} \frac{\dx}{a(t,x)^p}\le 0.
$$

$ii)$ For any positive constant $c$, if initially
$$
\forall x\in\xT^n,\qquad a(0,x)\ge c>0,
$$
then
\be\label{n151}
a(t,x)\ge c,
\ee
for all $(t,x)$ in $[0,T]\times\xT^n$.
\end{theo}
\begin{proof}
$i)$ Recall that
$$
\partial_t B-V\cdot \nabla B+(1-B)G(h)B=\gamma,
$$
so $a=1-B$ solves
$$
\partial_t a-V\cdot \nabla a +a G(h)a+\gamma=0.
$$
We have seen in Proposition~\ref{Coro:Zaremba-Taylor} that~$a(t,x)$ is positive 
for all $(t,x)$ in $[0,T]\times \xT^n$. 
Set
\be\label{defi:mMa}
m\defn \inf_{[0,T]\times\xT^n}a(t,x),\quad M\defn \sup_{[0,T]\times\xT^n}a(t,x),
\ee
and
$$
\zeta=\frac{1}{a}.
$$
Since $B$ is smooth, the function $\zeta$ is smooth 
and one verifies that
\be\label{eqzeta}
\partial_t\zeta-V\cdot \nabla \zeta=\zeta G(h)a+\gamma \zeta^2.
\ee
Our goal is to prove that, for all $p\ge 1$,
\be\label{nBpgoal}
\fract \int \zeta^p\dx \le 0.
\ee
To do so, we multiply the equation~\e{eqzeta} by $p\zeta^{p-1}$ to obtain
$$
\partial_t\zeta^p-V\cdot \nabla \zeta^p-p \zeta^{p} G(h)a-p\gamma \zeta^{p+1}=0.
$$
Then we integrate over $\xT^n$ and integrate by parts in the term $\int V\cdot \nabla \zeta^p\dx$. This gives that
$$
\fract \int \zeta^p\dx +\int (\cnx V)\zeta^p \dx -p \int \zeta^p G(h)a\dx -p\int \gamma \zeta^{p+1}\dx=0.
$$
Since $\zeta>0$ and since $\gamma\le 0$, one has
\be\label{nBpn1}
-p\int \gamma \zeta^{p+1}\dx\ge 0.
\ee
Let us prove that
\be\label{nBpgamma}
\int (\cnx V)\zeta^p \dx -p \int \zeta^p G(h)a\dx\ge 0.
\ee
By combining this inequality with~\e{nBpn1}, this will imply the 
desired result~\e{nBpgoal}. To prove \e{nBpgamma}, again we use the identity $G(h)B=-\cnx V$. 
This implies that $G(h)a=\cnx V$ and hence
$$
\int (\cnx V)\zeta^p \dx -p \int \zeta^p G(h)a\dx=(1-p)\int  \zeta^p G(h)a  \dx.
$$
If $p=1$ then the term in the right-hand side vanishes and the proof is complete. 
Otherwise $p>1$ and we can find a $C^\infty$ convex function $\Phi\colon\xR\to\xR$ such that
$$
\forall r\in [m,M],\qquad \Phi(r)=r^{-p+1},
$$
where $m$ and $M$ are given by \e{defi:mMa}. 
We are now in position to apply the convexity inequality~\e{convex}. 
This gives that
\begin{align*}
-\zeta^p G(h)a&=-a^{-p}G(h)a=\frac{1}{p-1}\Phi'(a)G(h)a\\
&\ge \frac{1}{p-1}G(h)(\Phi(a))= \frac{1}{p-1}G(h)(a^{-p+1}).
\end{align*}
Therefore, one can write that
$$
\int  -\zeta^p G(h)a  \dx\ge \frac{1}{p-1}\int G(h)(a^{-p+1})\dx=0,
$$
where we used the fact that $\int G(h)\psi\dx=0$ for any function $\psi$, 
which in turn follows from the divergence theorem:
\be\label{n140}
\int_{\xT^n} G(h)\psi\dx=\int_{\partial\Omega} \partial_n \varphi\dsigma=\iint_\Omega \Delta \varphi \dy\dx=0,
\ee
where we used the notations in~\e{defi:varphipsi}. This proves that
$$
\fract \int_{\xT^n} \frac{\dx}{a(t,x)^p} \le 0.
$$

$ii)$ Now, let us assume that initially $a(0,x)\ge c$ for some~$c>0$. 
Then, for any $p\ge 1$, one has
$$
\int_{\xT^n} \frac{\dx}{a(t,x)^p}\le\int_{\xT^n} \frac{\dx}{a(0,x)^p} 
\le \frac{\la \xT^n\ra}{c^p}. 
$$
Given $0<\delta<1$, introduce the set 
$A=\{ x\in \xT^n\,:\, a(t,x)< \delta c\}$ and denote its measure by $\la A\ra$. Then
$$
\int_{\xT^n} \frac{\dx}{a(t,x)^p}\ge \int_{A} \frac{\dx}{a(t,x)^p}\ge \int_{A} \frac{\dx}{(\delta c)^p}=\frac{\la A\ra}{(\delta c)^p}.
$$
By combining the two inequalities we get 
$\la A\ra\le \delta^p\la \xT^n\ra$ for any $p\ge 1$. 
Since $\delta<1$, this proves that $\la A\ra=0$ and hence $A=\emptyset$ since $A$ is open. 
This implies that $a(t,x)\ge c$ for all $(t,x)$ in $[0,T]\times \xT^n$, which completes the proof.
\end{proof}

\section{Decay estimate for the slope}\label{S:Lp2}

In this section, we prove Proposition~\ref{T:Vp}. 
To do so, we shall exploit the following conservation law which holds in any space dimension.
\begin{prop}\label{P4.5}
Assume that $h\in C^\infty([0,T]\times \xT^n)$ satisfies $\partial_t h+G(h)h=0$. 
Then, for any integer $p\in \xN\setminus\{0\}$, there holds  
\be\label{nV2p}
\fract \int \frac{\la V\ra^{2p}}{1-B}\dx
+ \int \bigg(2p \la V\ra^{2p-2}V\cdot G(h)V+(2p-1)\frac{\gamma  \la V\ra^{2p}}{(1-B)^2}\bigg)\dx=0.
\ee
\end{prop}
\begin{proof}
It follows from the chain rule that
\begin{align*}
&\frac{\la V\ra^{2p-2}V}{1-B}\cdot \big(\partial_t V-V\cdot \nabla V\big)\\
&\qquad\qquad=\frac{1}{2p(1-B)}\big(\partial_t -V\cdot \nabla\big)\la V\ra^{2p}\\
&\qquad\qquad=\frac{1}{2p}\big(\partial_t -V\cdot \nabla\big)\frac{\la V\ra^{2p}}{1-B}
-\frac{\la V\ra^{2p}}{2p(1-B)^2}\big(\partial_t -V\cdot \nabla\big)B.
\end{align*}
Thus, using the equation for $B$ and recalling that
$$
G(h)B=-\cnx V,
$$
we deduce that
$$
\partial_t B-V\cdot \nabla B =\gamma-(1-B)G(h)B=\gamma+(1-B)\cnx V.
$$
Consequently,
$$
\frac{\la V\ra^{2p-2}V}{1-B}\cdot \big(\partial_t V-V\cdot \nabla V\big)
=\frac{1}{2p}\Big(\big(\partial_t -V\cdot \nabla\big)\frac{\la V\ra^{2p}}{1-B}-\frac{\gamma \la V\ra^{2p}}{(1-B)^2}
-\frac{\la V\ra^{2p}}{(1-B)}\cnx V\Big),
$$
which yields
$$
\frac{\la V\ra^{2p-2}V}{1-B}\cdot\big(\partial_t V-V\cdot \nabla V\big)
=\frac{1}{2p}\bigg(\partial_t\frac{\la V\ra^{2p}}{1-B} -\cnx \Big(\frac{\la V\ra^{2p}V}{1-B}\Big)-\frac{\gamma \la V\ra^{2p}}{(1-B)^2}
\bigg).
$$
On the other hand, using the equation for $V$, one has
$$
\partial_t V-V\cdot \nabla V =-(1-B)G(h)V-\frac{\gamma}{1-B} V,
$$
which implies that
$$
\frac{\la V\ra^{2p-2}V}{1-B}\cdot\big(\partial_t V-V\cdot \nabla V\big)=-\la V\ra^{2p-2}V \cdot G(h)V
-\frac{\gamma \la V\ra^{2p}}{(1-B)^2}.
$$
By combining these two formulas we obtain that
$$
\frac{1}{2p}\partial_t\frac{\la V\ra^{2p}}{1-B}
-\frac{1}{2p}\cnx \Big(\frac{\la V\ra^{2p}V}{1-B}\Big)
-\frac{1}{2p}\frac{\gamma \la V\ra^{2p}}{(1-B)^2}=
-\la V\ra^{2p-2}V \cdot G(h)V
-\frac{\gamma \la V\ra^{2p}}{(1-B)^2},
$$
so
$$
\frac{1}{2p}\partial_t\bigg(\frac{\la V\ra^{2p}}{1-B}\bigg)
+\la V\ra^{2p-2}V \cdot G(h)V+\left(1-\frac{1}{2p}\right)\frac{\gamma \la V\ra^{2p}}{(1-B)^2}
=\frac{1}{2p}\cnx \bigg(\frac{\la V\ra^{2p}V}{1-B}\bigg).
$$
We deduce the desired result~\e{nV2p} by integrating in $x$ 
the previous identity.
\end{proof}

We are now in position to prove decay estimates for the $L^p$-norms of $V$. 
\begin{prop}\label{T:VLpbis}
Assume that $n=1$. Let $h\in C^\infty([0,T]\times \xT)$ be a solution to $\partial_th+G(h)h=0$.

\begin{enumerate}[i)]
\item If 
\be\label{assu:slopep=1}
\forall x \in \xT,\qquad  \la h_x(0,x)\ra\le 1,
\ee
where $h_x=\partial_xh$, then, for all time $t$ in $[0,T]$,
\be\label{n134}
\fract \int_{\xT} \frac{V(t,x)^{2}}{1-B(t,x)}\dx \le 0.
\ee
\item\label{T:Vpii} Consider an even integer $p\in 2(\xN\setminus\{0\})$. If
\be\label{assu:slope}
\forall x \in \xT,\qquad \la h_x(0,x)\ra\le \sqrt{\frac{p}{3p-2}},
\ee
then, for all time $t$ in $[0,T]$, there holds
\be\label{n141}
\fract \int_{\xT} \frac{V(t,x)^{2p}}{1-B(t,x)}\dx +2\int_{\xT} V^p G(h)(V^p)\dx\le 0.
\ee
\end{enumerate}
\end{prop}
\begin{proof}
Recall that the strategy of the proof is explained in~\S\ref{S:convexintro}, and that
the bounds \eqref{assu:slopep=1} and \eqref{assu:slope} immediately extend to all times $t \in [0,T]$
in view of Proposition \ref{prop:bornederiv}. 
The first key step is then to obtain a new identity relating $\gamma$ and~$V$. 
This is where we need to restrict the problem to space dimension $n=1$.
\begin{lemm}
If $n=1$, then there holds
\be\label{n1GBGV}
G(h)(B^2)-2BG(h)B=G(h)(V^2)-2VG(h)V,
\ee
and
$$
\gamma=\frac{2}{1+h_x^2}\big(G(h)(V^2)-2VG(h)V\big).
$$
\end{lemm}
\begin{proof}
Since 
$$
\gamma=\frac{1}{1+h_x^2}\big(G(h)(B^2+V^2)-2BG(h)B-2VG(h)V\big),
$$
the identity for $\gamma$ is a straightforward consequence of~\e{n1GBGV}. 

We now prove~\e{n1GBGV}. Denote by $\phi$ the harmonic extension of 
$h$ (so that $\Delta\phi=0$ in $\Omega:=\{y<h\}$ and $\phi\ah=h$). Then, 
by definition (see~\e{defi:BVphii}), one has $V=\phi_x\ah$ 
and $B=\phi_y\ah$, where $\phi_x=\partial_x\phi$ and 
$\phi_y=\partial_y \phi$. Introduce $\varphi=\phi_x^2-\phi_y^2$. Since $\varphi$ is the real part of the holomorphic function 
$(\phi_x(x,y)+i\phi_y(x,y))^2$, it is harmonic:
$$
\Delta_{x,y}\varphi=0.
$$
On the other hand, one has
$$
\varphi\ah=V^2-B^2.
$$
It follows that 
$$
G(h)(V^2-B^2)=((\partial_y-h_x\partial_x)\varphi)\ah.
$$
Consequently, using the chain rule, one finds that
$$
G(h)(V^2)-G(h)(B^2)=((\partial_y-h_x\partial_x)(\phi_x^2-\phi_y^2))\ah=
2VG(h)V-2BG(h)B,
$$
which completes the proof.
\end{proof}

We now prove the main result. We begin by recalling that, when $n=1$, the conservation law~\e{nV2p} reads
$$
\fract \int \frac{V^{2p}}{1-B}\dx+\Sigma_1+\Sigma_2=0,
$$
where
$$
\Sigma_1=2p \int V^{2p-1} G(h)V\dx,\quad 
\Sigma_2=(2p-1)\int\frac{\gamma  V^{2p}}{(1-B)^2}\dx.
$$
We want to prove that, if $p=1$ then $\Sigma_1+\Sigma_2\ge 0$ and if $p\in 2(\xN\setminus\{0\})$ then
\be\label{n142}
\Sigma_1+\Sigma_2\ge 2\int V^p G(h)V^p\dx.
\ee
Since
$$
\gamma=\frac{2}{1+h_x^2}\big(G(h)(V^2)-2VG(h)V\big),
$$
we have
$$
-\Sigma_2=(2p-1)\int \frac{2}{1+h_x^2}\big(2VG(h)V-G(h)(V^2)\big)\frac{V^{2p}}{(1-B)^2}\dx.
$$
By definition one has $V=(1-B)h_x$, thus one may write
$$
\frac{2}{1+h_x^2}\frac{V^{2p}}{(1-B)^2}=\frac{2V^{2}}{(1-B)^2(1+h_x^2)}V^{2p-2}=\frac{2h_x^2}{1+h_x^2}V^{2p-2},
$$
to obtain
$$
-\Sigma_2=(2p-1)\int \frac{2h_x^2}{1+h_x^2}\big(2VG(h)V-G(h)(V^2)\big)V^{2p-2}\dx.
$$
Now the key point is that, in light of~\e{convex}, one has the pointwise bound
$$
2VG(h)V-G(h)(V^2)\ge 0.
$$
So, if we set
$$
A\defn\sup_{[0,T]\times \xT} \frac{2h_x^2}{1+h_x^2},
$$
then we infer that
$$
-\Sigma_2\le A(2p-1)\int \big(2VG(h)V-G(h)(V^2)\big)V^{2p-2}\dx.
$$
Now, the assumption~\e{assu:slope} on the slope implies that
$$
A<\frac{p}{2p-1},
$$
and hence,
$$
-\Sigma_2\le p\int \big(2VG(h)V-G(h)(V^2)\big)V^{2p-2}\dx,
$$
which yields, by definition of $\Sigma_1$,
\be\label{n143}
-\Sigma_2\le \Sigma_1-p\int G(h)(V^2)V^{2p-2}\dx.
\ee
Next, we use again the convexity inequality~\e{convex}. More precisely, if $p=1$ then 
it suffices to write that
$$
\int G(h)(V^2)V^{2p-2}\dx=\int G(h)(V^2)\dx=0,
$$
since $\int G(h)\psi\dx=0$ for any function $\psi$ (see~\e{n140}). 
This proves that $\Sigma_1+\Sigma_2\ge 0$ when $p=1$. 
We now prove \e{n142} assuming that $p\in 2(\xN\setminus\{0\})$. 
If $p=2$, then one has directly
$$
\int G(h)(V^2)V^{2p-2}\dx=\int V^2 G(h)(V^2)\dx\ge 0,
$$
where we used the fact that $\int \psi G(h)\psi\dx \ge 0$ for any function $\psi$ (see~\e{n217}). 
Otherwise, $p=2k$ for some integer $k\ge 2$ and hence one may consider the $C^2$-convex function~$\Phi\colon\xR\to\xR$ defined by
$$
\Phi(r)=\left\{\begin{aligned}
&r^{2}\quad &&\text{if } p=4,\\
&\max\{0,r^{p/2}\} &&\text{if }p\ge 6.
\end{aligned}
\right.
$$
Then for any $p\in 2(\xN\setminus\{0,1\})$, the inequality~\e{convex} yields
\begin{align*}
V^{p-2}G(h)V^2&=(V^2)^{p/2-1}G(h)(V^2)\\
&=\frac 2p\Phi'(V^2)G(h)(V^2)\\
&\ge \frac 2pG(h)\Phi(V^2)=\frac 2pG(h)(V^p).
\end{align*}
It follows that
$$
V^{2p-2}G(h)V^2=\left(V^p\right)  (V^{p-2})G(h)V^2
\ge \frac{2}{p}V^p G(h)V^p,
$$
which implies that
$$
\int V^{2p-2}G(h)V^2 \dx\ge  \frac{2}{p}\int V^p G(h)V^p\dx\ge 0.
$$
So the wanted inequality~\e{n142} follows from~\e{n143}. This concludes the proof.
\end{proof}

\section{On the Cauchy problem}\label{S:Cauchy}

In this section we prove Theorem~\ref{T:Cauchy}. 
For the sake of shortness, we shall only prove {\em a priori} estimates. 
The uniqueness is known to hold in a broader setting. 
The proof of the existence from {\em a priori} estimates follows by adapting the arguments given in full details in 
\cite{AL} for the Muskat equation (the two problems are similar 
since they both concern nonlinear parabolic equation of order $1$ with similar fractional diffusion). This section contains 
no new result or argument. In fact 
our goal is precisely to show that one can solve the Cauchy problem by using results already proved in~\cite{ABZ3}. 
This is possible only because 
we work with the equations for the unknowns $B,V$.

\subsection{Microlocal analysis of the Dirichlet to Neumann operator}\label{S:paraDN}

For the reader's convenience, we recall in this subsection 
various results about the Dirichlet to Neumann operator. 

If $h$ is a $C^\infty$-function, it follows from classical elliptic regularity theory that, for any 
real number $s\ge 1/2$, $G(h)$ is bounded from $H^s(\xT^n)$ into 
$H^{s-1}(\xT^n)$ (the limitation $s\ge 1/2$ comes from the fact that a function $\psi$ in $H^{1/2}(\xT^n)$ is the trace of a $H^1$ function $\phi$ 
in the fluid domain $\Omega=\{y<h(x)\}$, so that 
$G(h)\psi$ is well-defined in $H^{-1/2}(\xT^n)$ by standard variational arguments). 
This property still holds in 
the case where $h$ has limited regularity. Namely, 
for $s> 1+n/2$ and $\sigma\in [1/2,s]$, we have (see \cite{CN,WuInvent, WuJAMS,LannesJAMS} and \cite[Theorem 3.12]{ABZ3})
\be\label{n1100}
\lA G(h)\psi\rA_{H^{\sigma-1}}\le C\big(\lA h\rA_{H^{s}}\big)\lA \psi\rA_{H^{\sigma}}.
\ee

On the other hand, it is known since the work of Calder\'on that, for 
$h \in C^{\infty}(\xT^n)$, 
$G(h)$ is a pseudo-differential operator. For the sake of completness, 
recall the definition of a pseudo-differential operator $\Op(a)$ with symbol $a=a(x,\xi)$. 
Firstly, given a function $u$ in the Schwartz space, 
we define the action of  
the pseudo-differential operator $\Op(a)$ on $u$ by
$$
\Op(a)u(x)=\frac{1}{(2\pi)^{n}}\int e^{ix\cdot\xi}a(x,\xi)\widehat{u}(\xi)\dxi.
$$
Then, by a duality argument, 
$\Op(a)$ extends as a continuous operator defined 
on the space of tempered distributions, which includes the Sobolev spaces of periodic functions. 
Then one has
$$
G(h)\psi = \Op(\lambda) \psi + R_0(h)\psi,
$$
where 
\begin{equation}\label{defi.lambda1}
\lambda=\sqrt{(1+|\nabla h| ^2)\la \xi\ra^2 - (\nabla h\cdot \xi)^2},
\end{equation}
and where the remainder satisfies the following property : there exists $K$ such that, for all $s\ge 0$,
$$
\lA R_0(h)\psi\rA_{H^{s}}\le C\left( \lA h\rA_{H^{s+K}}\right)\lA \psi\rA_{H^s}.
$$
This allows to approximate $G(h)$ by $\Op(\lambda)$ 
which is an operator of order $1$, modulo the remainder $R_0(h)$ 
which is of order $0$. 
Actually, we have an approximation at any order (see \cite{ABB,SU}).

On the other hand, 
notice that the symbol $\lambda$ is well-defined for 
any~$C^1$ function~$h$. It is therefore interesting to try 
to compare $G(h)$ and $\Op(\lambda)$ when 
$h$ has a limited regularity. This is possible thanks to Bony~\cite{Bony} 
paradifferential calculus (using in addition Alinhac's paracomposition operators 
and in particular the use of the so-called {\em good unknown of Alinhac}, following~\cite{Alipara}). 
The first results in this direction are due to 
Alazard and M\'etivier~\cite{AM}, and Alazard, Burq and Zuily~\cite{ABZ1,ABZ3}, following earlier work 
by~Lannes~\cite{LannesJAMS}. In particular, it was proved in \cite{ABZ3} that one can compare $G(h)$ to an explicit operator for any $h$ which is, by Sobolev embedding, in the H\"older space $C^{1+\epsilon}$ for some $\epsilon>0$. 
To introduce this result, we need to recall the definition of paradifferential operators. 
In our case, it will be simple since we need only a few results from that theory.

Recall that, for $\rho\in (0,1)$, 
we denote by $C^{\rho}(\xT^n)$ the 
H\"older space of bounded functions which 
are uniformly H\"older continuous with 
exponent $\rho$.

\begin{defi}
Given a real number $\rho\in (0,1)$ and $m\in\xR$, 
$\dot\Gamma_{\rho}^{m}$ 
denotes the space of
symbols $a(x,\xi)$
on $\xT^n\times(\xR^n\setminus 0)$ 
which are homogeneous of degree~$m$ and 
$C^\infty$ with respect to $\xi\neq 0$, and
such that, for all $\alpha\in\xN^n$ and all $\xi\neq 0$, the function
$x\mapsto \partial_\xi^\alpha a(x,\xi)$ belongs to $C^{\rho}(\xT^n)$ 
and
\begin{equation}\label{para:10}
\sup_{|\xi|=1}\lA \partial_\xi^\alpha a(\cdot,\xi)\rA_{C^{\rho}}<+\infty.
\end{equation}
\end{defi}

Now fix a cut-off function $\zeta$ such that $\zeta=0$ on a neighborhood of the origin 
and $\zeta=1$ for $|\xi|\ge 1$. Then introduce 
a $C^\infty$ function $\chi$ homogeneous of degree $0$ and satisfying, 
for $0<\eps_1<\eps_2$ small enough,
$$
\chi(\theta,\eta)=1 \quad \text{if}\quad \la\theta\ra\le \eps_1\la \eta\ra,\qquad
\chi(\theta,\eta)=0 \quad \text{if}\quad \la\theta\ra\geq \eps_2\la\eta\ra.
$$ 
Given a symbol $a$, we define
the paradifferential operator $T_a$ by
\begin{equation}\label{eq.para}
\widehat{T_a u}(\xi)=(2\pi)^{-n}\int \chi(\xi-\eta,\eta)\widehat{a}(\xi-\eta,\eta)\zeta(\eta)\widehat{u}(\eta)
\deta,
\end{equation}
where
$\widehat{a}(\theta,\xi)$ 
is the Fourier transform of $a$ with respect to the $x$ variable.

We need only to know the following properties of paradifferential operators. 

\begin{theo}\label{theo:sc}
\begin{enumerate}
\item\label{T:C}
If~$a \in \dot\Gamma^m_0$, 
then~$T_a$ is of order~$\leo m$ (it is bounded from $H^\mu$ into $H^{\mu-m}$ for all $\mu$ in $\xR$).

\item\label{T:PSC}Let $\rho\in (0,1)$. 
If~$a\in\dot\Gamma^{m}_{\rho}, b\in \dot\Gamma^{m'}_{\rho}$ then 
$T_a T_b -T_{a b}$ is of order~$\leo m+m'-\rho$ (it is bounded from $H^\mu$ into 
$H^{\mu-m-m'+\rho}$ for all $\mu$).

\item\label{T:paraprod} 
Consider three real numbers $r,\mu,\gamma$ satisfying
$$
r+\mu>0,\quad \gamma\le r  \quad\text{and}\quad \gamma < r+\mu-\frac{n}{2}.
$$
Then for any function $a=a(x)$ (depending only on $x$) and any $u=u(x)$, 
$$
a\in H^{r}\quad\text{and}\quad u\in H^{\mu}\quad\Rightarrow\quad au - T_a u\in H^{\gamma}.
$$
\end{enumerate}
\end{theo}

As explained above, the paradifferential calculus allows to compare $G(h)$ to an explicit operator. 
For our purposes, it will suffice to use the following

\begin{prop}[from~\cite{ABZ3}]\label{coro:paraDN1}
Let~$n\ge 1$. Consider real numbers $s,\sigma,\eps$ such that
$$
s>1+\frac{n}{2},\qquad \frac{1}{2}\leq \sigma \leq s-\mez,\qquad 
0<\eps\leq \mez, \qquad \eps< s-1-\frac{n}{2}.
$$
Then there exists a non-decreasing function~$\mathcal{F}\colon\xR_+\rightarrow\xR_+$ such that 
$$
R(h)f\defn G(h)f-T_\lambda f
$$
satisfies
\begin{equation*}
\lA R(h)f\rA_{H^{\sigma-1+\eps}(\xT^n)}\le \mathcal{F} \bigl(\| h \|_{H^{s}(\xT^n)}\bigr)\lA f\rA_{H^{\sigma}(\xT^n)}.
\end{equation*}
\end{prop}
\begin{rema}
In \cite{ABZ3} this result is proved for Sobolev spaces over $\xR^n$ 
but the same proof applies for periodic functions.
\end{rema}

\subsection{Paralinearization of the Hele-Shaw equation}

Inspired by the analysis for the water wave problem (\cite{AM,ABZ1,ABZ3}) or the Muskat equation~(\cite{AL}), 
we study the Cauchy problem for the Hele-Shaw equation by 
paralinearizing the latter equation. To do this, the trick is to work with the equations for $B$ and $V$ 
instead of using the equation for $h$. 
By so doing, we are led to consider a case where the Dirichlet-to-Neumann operator $G(h)$ is applied on a function which is $1/2$-derivative less regular than $h$. 

We want to prove {\em a priori} estimates for Sobolev norms of $h$. 
Fix $T>0$ and set $I=[0,T]$. We consider a solution $h\in C^1(I;H^s(\xT^n))$ to $\partial_t h+G(h)h=0$. 
To prove parabolic estimates, we introduce the following spaces: given~$\mu\in\xR$, set
\begin{equation}\label{XY}
X^\mu(I)=C^0(I;H^{\mu}(\xT^n))
\cap L^2(I;H^{\mu+\mez}(\xT^n)).
\end{equation}

We want to estimate $\lA h\rA_{X^s(I)}$. 
Since we have an estimate in $X^0(I)$ by using the energy inequality~\e{Energy}, 
$$
\mez\fract\int h^2\dx +\int h G(h)h \dx =0,
$$
it will suffice to estimate the $X^{s-1}(I)$-norm of $\nabla h=\nabla_x h$. To do so, we exploit the fact that
\be\label{n2001}
\nabla h =\frac{V}{1-B}.
\ee
Recall that the Rayleigh-Taylor coefficient $1-B$ is positive so that we can divide by $1-B$. 
By compactness of $\xT^n$, this coefficient is bounded from below by a positive constant $\underline{a}$ at initial time and hence, by a continuity argument, we can 
assume that $a$ is bounded from below by $\underline{a}/2$ to 
prove {\em a priori} estimates\footnote{We do not need  to (and in fact we 
cannot) apply directly the conclusion of Theorem~\ref{T:Bp} to infer that $1-B(t,x)\ge \underline{a}$ 
for all time $t$. This is because we do not have a similar result for an iterative scheme converging to the solution.}. Now, since $s-1>n/2$, the Sobolev spaces $H^{s-1}(\xT^n)$ and $H^{s-\mez}(\xT^n)$ are algebras, and we have the classical Moser tame estimates for $\sigma>n/2$,
\be\label{Moser}
\lA uv\rA_{H^\sigma}\les \lA u\rA_{L^\infty}\lA v\rA_{H^\sigma}+\lA v\rA_{L^\infty}\lA u\rA_{H^\sigma}.
\ee
This easily implies that, to estimate $\nabla h$ in $X^{s-1}(I)$, 
it will suffice to estimate $B$ and $V$ and $X^{s-1}(I)$. 
To do this, we shall paralinearize their equations.

\begin{lemm}Let~$n\ge 1$ and consider real numbers $s,\eps$ such that
$$
s>1+\frac{n}{2},\qquad 
0<\eps\leq \mez, \qquad \eps< s-1-\frac{n}{2}.
$$
Assume that $h$ is a smooth solution of the Hele-Shaw equation. Then 
$$
\partial_t B-T_V\cdot\nabla B+T_{a\lambda }B=F_1,
$$
and 
$$
\partial_t V-T_V\cdot\nabla V+T_{a\lambda }V=F_2
$$
where, for any time $t$,
\be\label{n1008}
\lA F_1(t)\rA_{H^{s-\tdm+\eps}}+\lA F_2(t)\rA_{H^{s-\tdm+\eps}}\le \mathcal{F}\big(\lA h(t)\rA_{H^s}\big)\big(1+\lA (B(t),V(t))\rA_{H^{s-\mez}}\big),
\ee
for some nondecreasing function $\mathcal{F}\colon \xR_+\rightarrow\xR_+$.
\end{lemm}
\begin{rema}
This means that $B$ and $V$ solve parabolic evolution equations of order $1$ with remainder terms of order $1-\eps$. 
These remainder terms are harmless since they can be absorbed by classical interpolation arguments and energy estimates.
\end{rema}
\begin{proof}
We say that $F$ is an admissible remainder provided that the $H^{s-\tdm+\eps}(\xT^n)$-norm of $F(t)$ is bounded by 
$\mathcal{F}\big(\lA h(t)\rA_{H^s}\big)\big(1+\lA (B(t),V(t))\rA_{H^{s-\mez}}\big)$ (since all our estimates are pointwise in time, 
we will skip the time dependence in this proof). Given two expressions $A_1,A_2$ 
depending on $h$, we write $A_1\sim A_2$ 
to say that $A_1-A_2$ is an admissible remainder. 
We shall make extensive use of the estimate
\be\label{n1215}
\lA B\rA_{H^{s-1}}+\lA V\rA_{H^{s-1}}
\le \mathcal{F}\big(\lA h\rA_{H^s}\big)\lA h\rA_{H^s},
\ee
which follows, for instance, from the relations
$$
B= \frac{G(h)h+\la \partialx h\ra^2}{1+|\partialx  h|^2},
\qquad
V=(1-B)\partialx h,
$$
using also the estimate~\e{n1100} for the Dirichlet to Neumann operator and the product rule~\e{Moser}.

Recall that~$B$, $V$ and $\gamma$ satisfy
\begin{align}
&\partial_t B-V\cdot \nabla B +aG(h)B=\gamma,\label{n1001}\\
&\partial_t V-V\cdot \nabla V +aG(h)V+\frac{\gamma}{1-B} V=0,\label{n1002}\\
&\gamma=\frac{1}{1+|\nabla h|^2}\Big(G(h)\big(B^2+|V|^2\big)-2BG(h)B-2V\cdot G(h)V\Big).\label{n1003}
\end{align}

Directly from the classical paralinearization formula for products 
(see point~\ref{T:paraprod} in Theorem~\ref{theo:sc} applied with $(r,\mu,\gamma)=(s-1,s-3/2,s-3/2+\eps)$), 
one has
$$
V\cdot \nabla B\sim T_V\cdot \nabla B,
$$
and
$$
V\cdot \nabla V\sim T_V\cdot\nabla V.
$$
Similarly, since \e{n1100} implies that $G(h)B$ and $G(h)V$ belong to 
$H^{s-3/2}(\xT^n)$, we have
$$
aG(h)B\sim T_a G(h)B,\quad aG(h)V\sim T_a G(h)V.
$$
On the other hand it follows from Proposition~\ref{coro:paraDN1} applied with $\sigma=s-1/2$, that
$$
G(h)B\sim T_\lambda B.
$$
Consequently, by using the 
classical results from paradifferential calculus (namely the continuity property of paradifferential operators and symbolic calculus, see points \ref{T:C} and \ref{T:PSC} in Theorem~\ref{theo:sc}), we 
successively verify that
$$
T_a G(h)B\sim T_a T_\lambda B\sim T_{a\lambda}B.
$$
Similarly, one has $aG(h)V\sim T_{a\lambda}V$. We thus have proved that
\begin{align*}
&\partial_t B-T_V\cdot\nabla B+T_{a\lambda }B\sim \gamma
&\partial_t V-T_V\cdot\nabla V+T_{a\lambda }V\sim -\frac{\gamma}{1-B} V.
\end{align*}
It remains only to prove that the above right-hand sides are equivalent to $0$. 
To do so, it suffices to prove that
$\gamma\sim 0$. Indeed, since $s-1>n/2$, 
$$
\lA \frac{\gamma}{1-B} V\rA_{H^{s-\tdm+\eps}}=
\lA \gamma \nabla h\rA_{H^{s-\tdm+\eps}}
\le \lA \gamma\rA_{H^{s-\tdm+\eps}}\lA \nabla h\rA_{H^{s-1}}
$$
and hence the relation $\gamma\sim 0$ will imply that $\frac{\gamma}{1-B} V
\sim 0$. So we only have to prove that $\gamma\sim 0$. 
In view of the definition of $\gamma$ and using again the product rule in Sobolev space, 
this will be a consequence of 
$$
G(h)\big(B^2+|V|^2\big)-2BG(h)B-2V\cdot G(h)V\sim 0.
$$
To prove the latter result, we use again the product rule in Sobolev spaces \e{Moser}, the bound \e{n1215} and Proposition~\ref{coro:paraDN1} to infer that
$$
G(h)\big(B^2+|V|^2\big)-2BG(h)B-2V\cdot G(h)V\sim T_\lambda \big(B^2+|V|^2\big)-2BT_\lambda B-2V\cdot T_\lambda V.
$$
We then paralinearize the products:
$$
B^2+|V|^2\sim 2T_B B+2T_V\cdot V,\quad BT_\lambda B\sim T_BT_\lambda B,\quad V\cdot T_\lambda V\sim T_V\cdot T_\lambda V.
$$
We conclude thanks to symbolic calculus (see point~\ref{T:PSC} in Theorem~\ref{theo:sc} applied with $(m,m',\rho)=(1,0,\eps)$) 
that
$$
[T_\lambda, T_B]B\sim 0,\quad [T_\lambda, T_V]\cdot V\sim 0.
$$
This terminates the proof of $\gamma\sim 0$, which completes the proof of the lemma.
\end{proof}

We are now in position to apply immediately another result proved in \cite{ABZ3} for 
paradifferential parabolic evolution equations. 
%

\begin{prop}(see Prop. $2.18$ in \cite{ABZ3})\label{prop:max}
Let~$r\in \xR$,~$\rho\in (0,1)$,~$I=[0,T]$ and consider a time dependent symbol 
$p=p(t;x,\xi)$ 
bounded from~$I$ 
into~$\dot{\Gamma}^1_\rho(\xT^n)$, that is
\begin{equation}\label{n1009}
\mathcal{M}^1_\rho(p)=\sup_{t\in I}
\sup_{\la\alpha\ra\le 2(n+2) +\rho  ~}\sup_{\la\xi\ra =1}
\lA \partial_\xi^\alpha p(t;\cdot,\xi)\rA_{C^{\rho}(\xT^n)}<+\infty,
\end{equation}
and
satisfying
\be\label{n1010}
\RE p(t;x,\xi) \geq c \la\xi\ra,
\ee
for some positive constant~$c$. Then for any source term
$$
f\in L^p(I;H^{r-1+\frac1p}(\xT^n))\quad\text{with}\quad p\in [1,+\infty),
$$
and any intial data~$w_0\in H^{r}(\xT^n)$, there exists~$w\in X^{r}(I)$ solution of  the parabolic evolution equation
\begin{equation}\label{eqW}
\partial_t w + T_p w =f,\quad w\arrowvert_{t=0}=w_0,
\end{equation}
satisfying 
\begin{equation*}
\lA w \rA_{X^{r}(I)}\le K\left\{\lA w_0\rA_{H^{r}}+ \lA f\rA_{L^p(I;H^{r-1+\frac1p}(\xT^n))}\right\},
\end{equation*}
for some positive constant~$K$ depending only on~$r,\rho,c$ and~$\mathcal{M}^1_\rho(p)$.
Furthermore, this solution is unique in ~$X^s(I)$ for any~$s\in \xR$. 
\end{prop}

We apply this proposition with
$$
p(t;x,\xi)=-iV(t,x)\cdot \xi+a(t,x)\lambda(t;x,\xi),
$$
where recall that 
$
\lambda=\sqrt{(1+|\nabla h| ^2)\la \xi\ra^2 - (\nabla h\cdot \xi)^2}$. 
Then \e{n1009} is clearly satisfied with $\rho=\eps$ 
and, since $\RE p=a\lambda$, 
the assumption~\e{n1010} is also satisfied. 

Now, the key point is that the estimate \e{n1008} for $(F_1,F_2)$ means 
$F_1$ and $F_2$ are estimated in 
$L^2(I;H^{s-\tdm+\eps}(\xT^n))$ in terms of the $X^{s-1}(I)$-norms of $B,V$. 
As a result, 
by using the previous proposition with $r=s-1$ and $p\in [1,2)$ chosen so that 
$$
r-1+\frac1p=s-\tdm+\eps,
$$
it follows from the H\"older inequality in time that there exists $\theta\in (0,1)$ (in fact $\theta=\eps$) 
such that 
$B$ and $V$ satisfy an {\em a priori} estimate of the form
\begin{align*}
\lA (B,V)\rA_{X^{s-1}(I)}&\le K \lA (B(0),V(0))\rA_{H^{s-1}(\xT^n)}\\
&\quad+T^\theta
\mathcal{F}\big(\lA h\rA_{L^\infty(I;H^s)}\big)\big(1+
\lA (B,V)\rA_{X^{s-1}(I)}\big).
\end{align*}
Then, as explained above (see the discussion following~\e{n2001}),
$$
\lA h\rA_{X^{s}(I)}\le K \lA h_0\rA_{H^{s}(\xT^n)}+T^\theta
\mathcal{F}\big(\lA h\rA_{L^\infty(I;H^s)}\big)\big(1+
\lA h\rA_{X^{s}(I)}\big).
$$
This shows that, for $T$ small enough, one has a uniform estimate for $h$. Which concludes the analysis.

\subsection*{Acknowledgements}
We thank the reviewers for their careful readings of the manuscript. 
T.A.\ and D.S.\ acknowledge the support of the SingFlows project, 
grant ANR-18-CE40-0027 of the French National Research Agency (ANR).

\clearpage


\noindent\textbf{Thomas Alazard}\\
\noindent Universit\'e Paris-Saclay, ENS Paris-Saclay, CNRS,\\
Centre Borelli UMR9010, avenue des Sciences, 
F-91190 Gif-sur-Yvette\\

\vspace{3mm}

\noindent\textbf{Nicolas Meunier}\\
\noindent Laboratoire de Math{\'e}matiques et Mod{\'e}lisation d'Evry (LaMME), UMR 
8071, Universit{\'e} d'Evry Val d'Essonne, 23 Boulevard de France 91037, Evry, France

\vspace{3mm}

\noindent\textbf{Didier Smets}\\
\noindent Laboratoire Jacques-Louis Lions, Sorbonne Universit\'e, 4 Place Jussieu 75005 Paris, France

\end{document}